\documentclass[journal]{IEEEtran} 

\usepackage{amsmath}
\usepackage{times}
\usepackage{amssymb}
\usepackage{amsthm}

\usepackage{graphicx}
\usepackage{tikz}
\usepackage{pdflscape}

\usepackage[ruled,norelsize]{algorithm2e}

\makeatletter
\newcommand{\removelatexerror}{\let\@latex@error\@gobble}
\makeatother

\newtheorem{Theorem}{Theorem}

\theoremstyle{definition}

\title{Sample-based Population Observers}

\author{Shen Zeng
\thanks{Shen Zeng is with the Department of Electrical and Systems Engineering at Washington University in St.\ Louis, Email: \texttt{s.zeng@wustl.edu}}
}

\begin{document}

\maketitle

\begin{abstract}
In this paper, a first sample-based formulation of the recently considered population observers, or ensemble observers,
which estimate the state distribution of dynamic populations from measurements of the output distribution is established. The results presented in this paper yield readily applicable computational procedures that are no longer subject to the curse of dimensionality, which all previously developed techniques employing a kernel-based approach are inherently suffering from. The novel insights that eventually pave the way for all different kinds of sample-based considerations are in fact deeply rooted in the basic probabilistic framework underlying the problem, bridging optimal mass transport problems defined on the level of distributions with actual randomized strategies operating on the level of individual points. The conceptual insights established in this paper not only yield insight into the underlying mechanisms of sample-based ensemble observers but significantly advance our understanding of estimation and tracking problems for the class of ensembles of dynamical systems in general. 


\end{abstract}

\section{Introduction}
\IEEEPARstart{T}{he} observability problem in systems theory systematically addresses a task fundamental to numerous scientific fields, particularly those close to physics, namely the extraction of information about the state of a dynamical process from knowledge of the underlying dynamics, and time series data of some less informative output measurement. The concept of observability together with the concept of controllability of a linear state-space model layed the basic foundation of a general theory of (control) systems (see \cite{Kalman1959_general_theory, Kalman1963_mathematical_linear_systems}), which has fundamentally reshaped the way we think about systems.
Out of this quite abstract approach, virtually as a side product of the deep systems theoretic undertakings of Kalman, the celebrated Kalman filter \cite{kalman1960filter} was born\footnote{Kalman himself had described his discovery a mere corollary of his much more encompassing state-space approach on different occasions.}, which since then has been a key enabling device for several important applications. 

The same line of thoughts centered around the questions of controllability and observability are recently being investigated in relation to a new class of systems, consisting of populations of dynamical systems of the same structure with a given distribution in their states \cite{brockett2000stochastic, brockett2007optimal, li2009ensemble, li2011ensemble, brockett2012notes, zeng2015tac, zeng2017sampled}. While a classical system can be thought of as a single point particle evolving in state-space (following the combined effect of a drift and a control vector field), for a population comprised of a large number of dynamical system, the point describing the state of the system would be replaced by a (probability) distribution of points, as suggested in Figure~\ref{fig:vector_field_density}.

\newpage
\begin{figure}[htp!]
	\centering
\vspace{0.15cm}
	\includegraphics[width=0.17\textwidth]{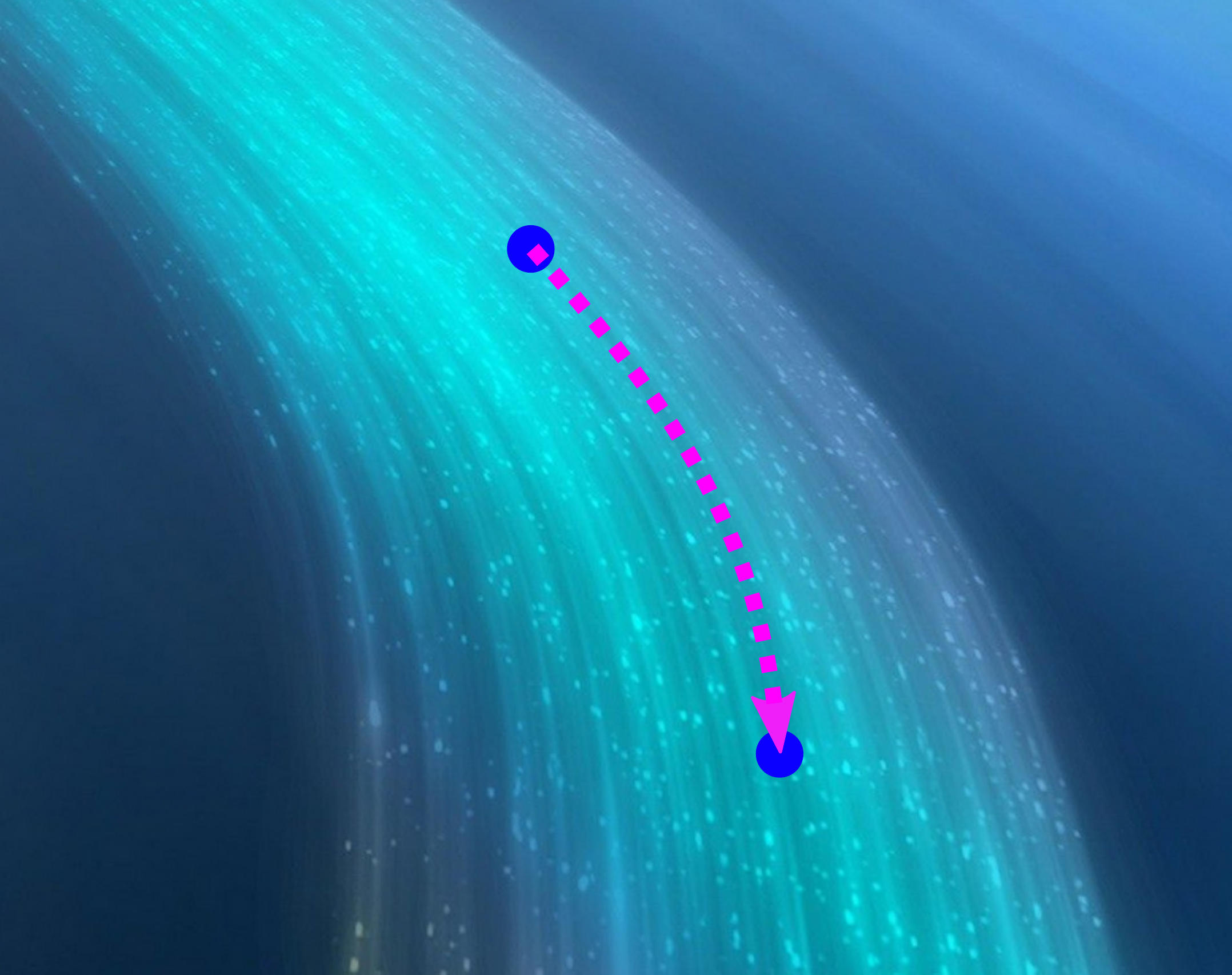} \;\;\;\;\;\; \;\;\;\; \includegraphics[width=0.17\textwidth]{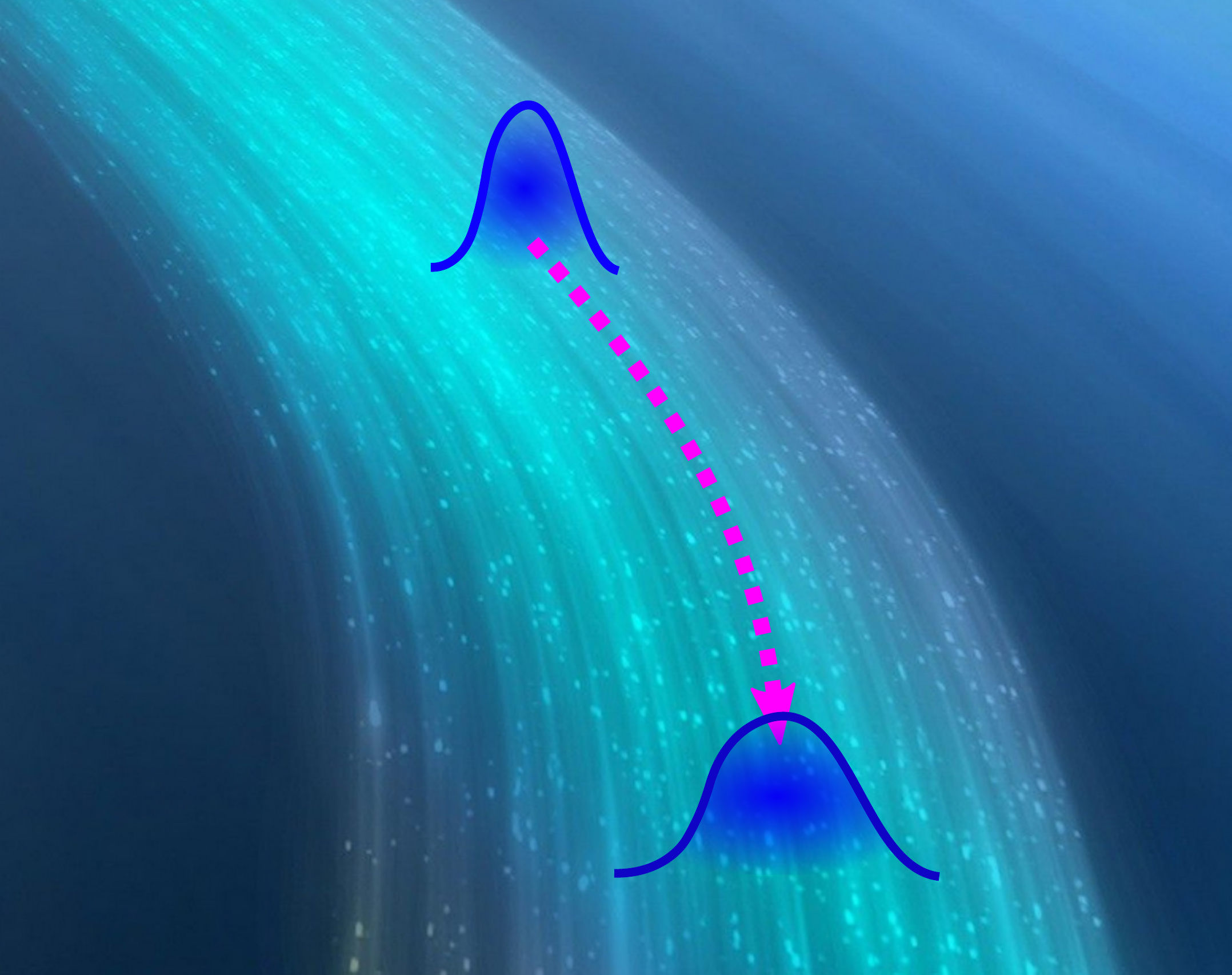}
\vspace{-0.1cm}
	\caption{The evolution of the state of a classical system is typically thought of as a point evolving in state space (left). In the same spirit, the dynamics of a population of systems is described by distributions of points (right).}
	\label{fig:vector_field_density}
\end{figure}

Of course, the idea of considering probability distributions as a description of the state of a system is not new -- in fact it traces back more than 100 years to the early beginnings of statistical mechanics, where the occurring probability distribution was already used both as a model for the state of one uncertain system or of an actual population of many systems, with a distribution in initial states. However, it has only recently become clear that once we look closer at the interface of really interacting with actual populations of systems, very distinct restrictions start to surface. This is where the probabilistic model splits into two branches, each with completely different interpretations with regard to what is being measured, and how we are able to exert control over the system.

A prime example that illustrates the fundamentally different interpration of the probabilistic setup for the situation of populations of dynamical systems is given by heterogeneous cell populations, such as cancer cell populations. For example, an important task for such heterogeneous cell populations is to estimate the specific distribution in states\footnote{The state of a single cell is typically described by the set of concentrations of different molecules or proteins, which are governed by regulatory networks that in turn can be described by ordinary differential equations.} or parameters, as such distribution can often be the key driver for heterogeneous responses to an external biochemical stimulus, like it is prominently observed with cancer, where we often see the survival of subpopulations during drug treatment. The given data for solving the estimation task are measurements of only a subset of molecule concentrations, which furthermore are increasingly being recorded via high-throughput devices called flow cytometers. By rapidly passing a stream of fluorescently labeled cells through a laser and fluorescence detectors, flow cytometers can easily gather concentration measurements of a vast number of cells. However, the ability to gather vast amounts of data comes at a cost. Namely, it is only possible to measure at the population level, which here, specifically, means that nothing can be said about an individual cell; it is only that a lot of measurements are being recorded and then stored in the form of histograms or other statistics. This circumstance may be described as a population-level observation, and is illustrated Figure~\ref{fig:aggregate_data}.

\begin{figure}[htpb]
\vspace{-0.3cm}
 \centering
    \includegraphics[width=0.43 \textwidth]{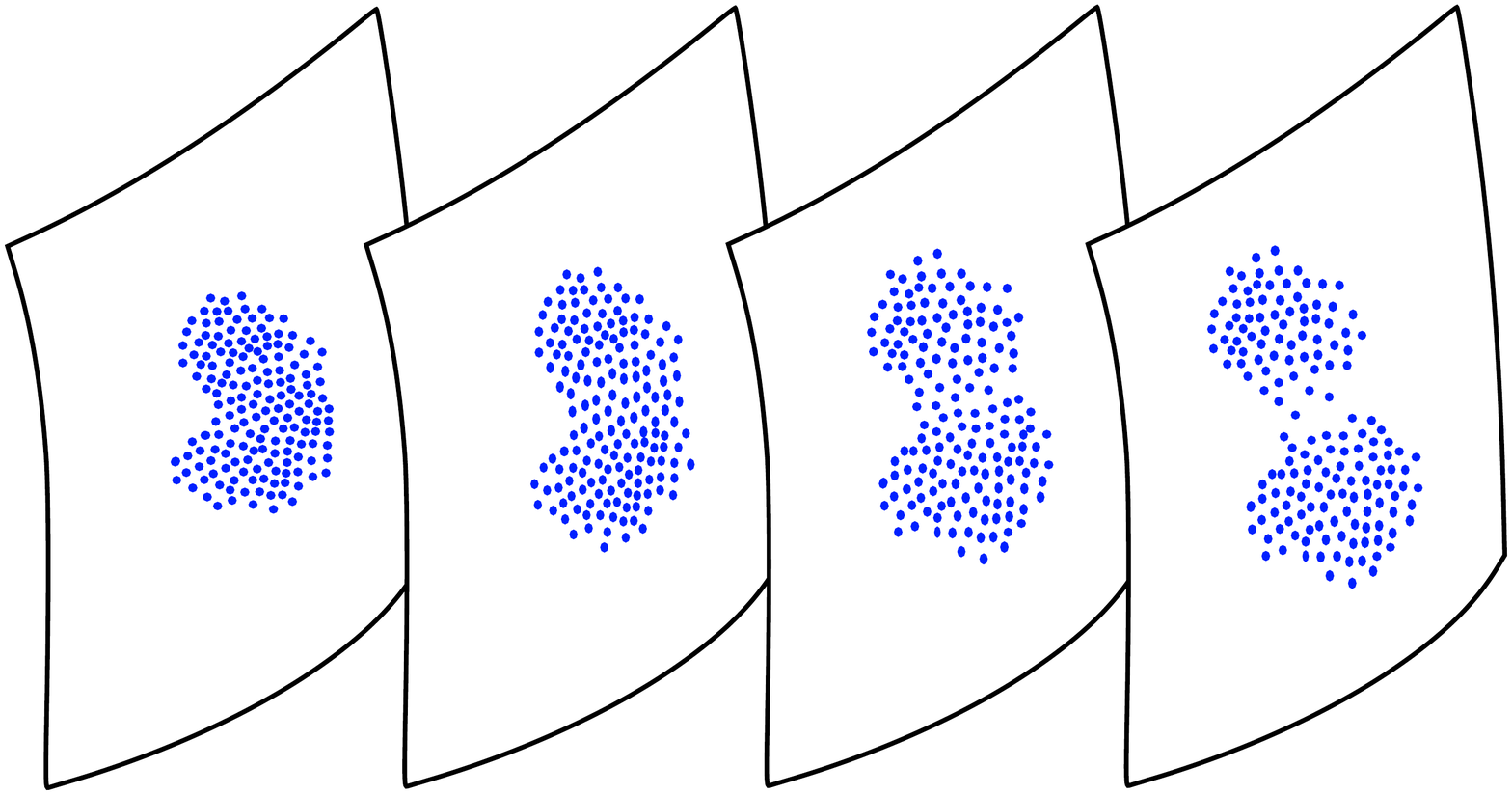} 
	\put(-222,112){\rotatebox{34.3}{Output space}}
	\put(-192,16){\rotatebox{34.3}{$t=t_1$}}
	\put(-142,16){\rotatebox{34.3}{$t=t_2$}}
	\put(-92,16){\rotatebox{34.3}{$t=t_3$}}
	\put(-40,16){\rotatebox{34.3}{$t=t_4$}}
	\vspace{-0.55cm}
      \caption{An illustration of population snapshots. In each time step $t_1, \dots, t_4$ we have a snapshot of certain output values of a population. The crucial point is that in a snapshot, information relating an output value to the individual producing that output value is completely missing. Taken from \cite{zeng2015tac}.}
\label{fig:aggregate_data}
\end{figure}

While Figure~\ref{fig:aggregate_data} may give the impression that one is measuring many output trajectories of individual cells, but without recording the actual associations between measurements in different time points, the situation is in fact even more cumbersome for the example of cell populations. This is because we only get to measure each cell once, due to the simple reason that after it is measured, it is either destroyed or gone. Therefore, the measurements at different snapshots may stem from completely different individuals in the population; they do, however, all stem from the same population. 

These considerations led us to view these population snapshots as samples from an output distribution, and to further view the output distribution as the ``total'' output of the population. This idea was then formalized in terms of a novel systems theoretic setup in which a classical system with output
\begin{align*}
\dot{x} &= f(x) \\
      y &= h(x),
\end{align*}
with $f: \mathbb R^n \to \mathbb R^n$ and $h: \mathbb R^n \to \mathbb R^m$, is generalized by means of a distribution in initial states. More specifically, the distribution of initial states in the population is modelled by a probability distribution, i.e. the initial state is taken to be a random vector $x(0) \sim \mathbb P_0$, with a non-parametric probability distribution $\mathbb P_0$. This clearly leads to a probabilistic nature of the output as well, which we describe in terms of $y(t) \sim \mathbb P_{y(t)}$. 

The practical ensemble observability problem consists of  reconstructing the initial state distribution $\mathbb P_0$ when given the evolution of the distribution of outputs $\mathbb P_{y(t)}$, which, again, is fundamentally different from classical filtering problems in which the measured data are \emph{single realizations} of the output distribution associated to a \emph{single uncertain} plant. In \cite{zeng2015tac}, we first studied the ensemble observability problem in the linear case, where $f(x)=Ax$ and $h(x)=Cx$, both from a theoretical and practical perspective. The investigations of the underlying basic theoretical problem in particular also revealed a deep connection between such ensemble observability problems and mathematical tomography problems, providing crucial insights into the inner systems theoretic mechanisms and, from a practical perspective, also immediately rendered the problem amenable to computational solutions. 

The computational solutions, however, having been very much anchored in the tomography-based considerations, were inevitably affected by the curse of dimensionality. While problems in tomography most prominently take place in lower dimensions, specifically dimensions two or three, in the ensemble observability setup, such a restriction is naturally undesirable, as the dimension of the state space is in general typically higher. In \cite{zeng2015nonlinear}, we already pointed out that in our quest to get better insights about the initial state distribution, we eventually want to circumvent the route over distributions, in which the output snapshots are first mapped into discretized distributions (histograms), from which a discretized initial state distribution is then to be reconstructed. Instead, a sample-based approach to the reconstruction problem was envisioned, which from a pragmatic standpoint seems very natural as well, as the measurement data is naturally given in terms of samples of the output in the first place, and not in terms of the distribution of the output, which is just a mathematical idealization introduced for the sake of studying the theoretical problem. In order to establish a sample-based framework, we need to derive a systematical procedure that takes the samples of the output distribution at different time points and eventually returns a set of points that could very well be samples from the initial state distribution. In other words, we seek for a procedure that lets us ``sample'' from the state distribution by directly utilizing the output snapshots.

From this perspective, this amounts to solving the probabilistic analogue of the tomography problem \emph{in a statistical framework}. More specifically, in this probabilistic setup, we really view the available Radon projections as marginal distributions (in a probabilistic sense) and the actual data that can be used not as the marginal distribution but rather samples from it. The question then would read as: How can we generate a set of points in $\mathbb R^n$ that best mimics a set of real samples of the sought joint distribution? Naturally, this problem formulation leads us to think of an approach in the spirit of \emph{Markov chain Monte Carlo methods}. In this paper, this idea is eventually realized by leveraging a connection to optimal mass transport problems \cite{villani2008,chen2017optimal}, which is in fact very fundamental, and leads to novel interesting theoretical insights and questions.

The structure of the paper is as follows. In Section~\ref{sec:review}, we provide a very brief review of the ensemble observability problem, with its many different viewpoints and connections to other areas of mathematics. In particular, we will discuss an example of a nonlinear observability problem, which already provides some important hints towards establishing a sample-based approach. This sample-based approach is then fully established in Section~\ref{sec:ensemble_estimator}, yielding both sample-based ensemble estimators and observers. All key steps in the introduction of the sample-based scheme are complemented by detailed illustrations and examples. In Section~\ref{sec:discrete_ensemble}, we turn towards the discrete ensemble observability problem \cite{zeng2017tac}, which is a problem closely related to the initially introduced general ensemble observability. We are able to accomplish our long-lasting effort to establish a unification of the continuous and discrete version of the ensemble observability problem, resulting in a coherent computational framework centered around the optimal mass transport formulation.

\section{The ensemble observability problem}
\label{sec:review}
In this section we provide a rapid review of different aspects of the ensemble observability problem that are most relevant to the presentation of the novel insights and results put forth of this paper. In particular, we will put significant emphasis on the discussion of the relation between the ensemble observability problem and mathematical tomography problems, established in \cite{zeng2015tac}, by which the ensemble observability problem also first became amenable to comprehensible computational solutions. 

We recall that in the general ensemble observability problem we ask under which conditions we can reconstruct the initial state distribution $\mathbb P_0$ when given the evolution of the distribution of outputs $\mathbb P_{y(t)}$, under a finite-dimensional (nonlinear) dynamical system. Furthermore, we are interested in practical reconstruction techniques for this problem. In \cite{zeng2015tac}, we first studied this problem in the linear case, both from a theoretical and practical perspective. To first build some intuition around the whole concept of ensemble observability, we consider an example with a two-dimensional harmonic oscillator
\begin{align*}
\dot{x} &= \begin{pmatrix} \phantom{-}0 & 1 \\ -1 & 0 \end{pmatrix}x, \\
y &= \begin{pmatrix} 1 & 0 \end{pmatrix} x.
\end{align*}
with a bimodal initial distribution as depicted in Figure~\ref{fig:harmonic_oscillator_setup}. The measured output distribution corresponding to the output $y = x_1$ of the underlying linear system results from a marginalization of the state distribution over the second coordinate, i.e.\ from integration along the $x_2$-direction. Thus, when the system evolves, the state distribution is subject to both a transportation with the flow, and a marginalization over the second coordinate, resulting in an evolution of the output distribution, as suggested in Figure~\ref{fig:harmonic_oscillator_setup}. 
\begin{figure}[htp!]
	\centering
	\includegraphics[width=0.49\textwidth]{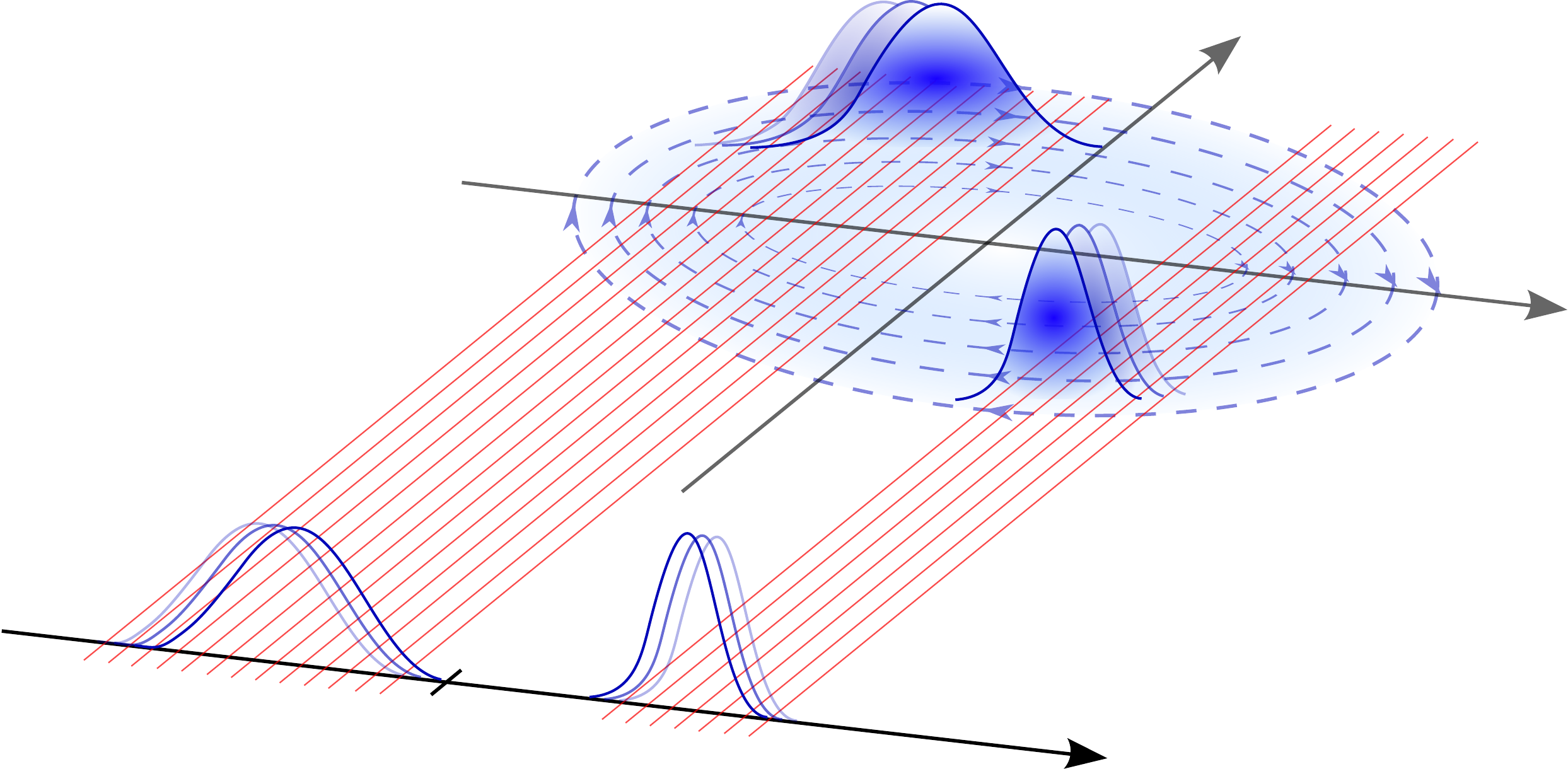}
	\caption{Illustration of the ensemble observability problem for a two-dimensional harmonic oscillator with a bimodal initial distribution. The upper right shows the evolution of the state distribution. The lower left shows the evolution of the corresponding output distribution. Taken from \cite{zeng2016diss}.}
	\label{fig:harmonic_oscillator_setup}
\end{figure}

The question in the ensemble observability problem for the specific example is thus whether or not one can reconstruct the (initial) state distribution from only observing the evolution of the output distribution, shown in the lower left of Figure~\ref{fig:harmonic_oscillator_setup}. Even though one might consider this a quite systems theoretic perspective on the problem, an answer to this problem is simply not immediate in this considered setting, which is a rather remarkable conclusion.

Due to the aforementioned reasons, in \cite{zeng2015tac} we took a different approach to the problem, which is to simply view and treat it as a (generic) inverse problem in a measure theoretic framework. In fact, the output distribution $\mathbb P_{y(t)}$ is related to the initial distribution $\mathbb P_0$ in a very basic way, namely through a pushforward relation
\vspace{-0.15cm}
\begin{align*}
\mathbb P_{y(t)} (B_y) := \mathbb P_0 ((Ce^{At})^{-1} (B_y))  =  \int_{(Ce^{At})^{-1} (B_y)} \;\, p_0(x) \; \text{d}x.
\end{align*}
The values of the output distribution are related to the initial density through these integrals over these preimages, which one can think of as a strips due linearity of $x \mapsto Ce^{At}x$, as well as the fact that the interesting cases occur only when $C$ does not have full column rank. This basic perspective may be illustrated as in Figure~\ref{fig:density_tomography}. 
\begin{figure}[htp!]
\centering
\vspace{-0.24cm}
\includegraphics[width=0.28 \textwidth]{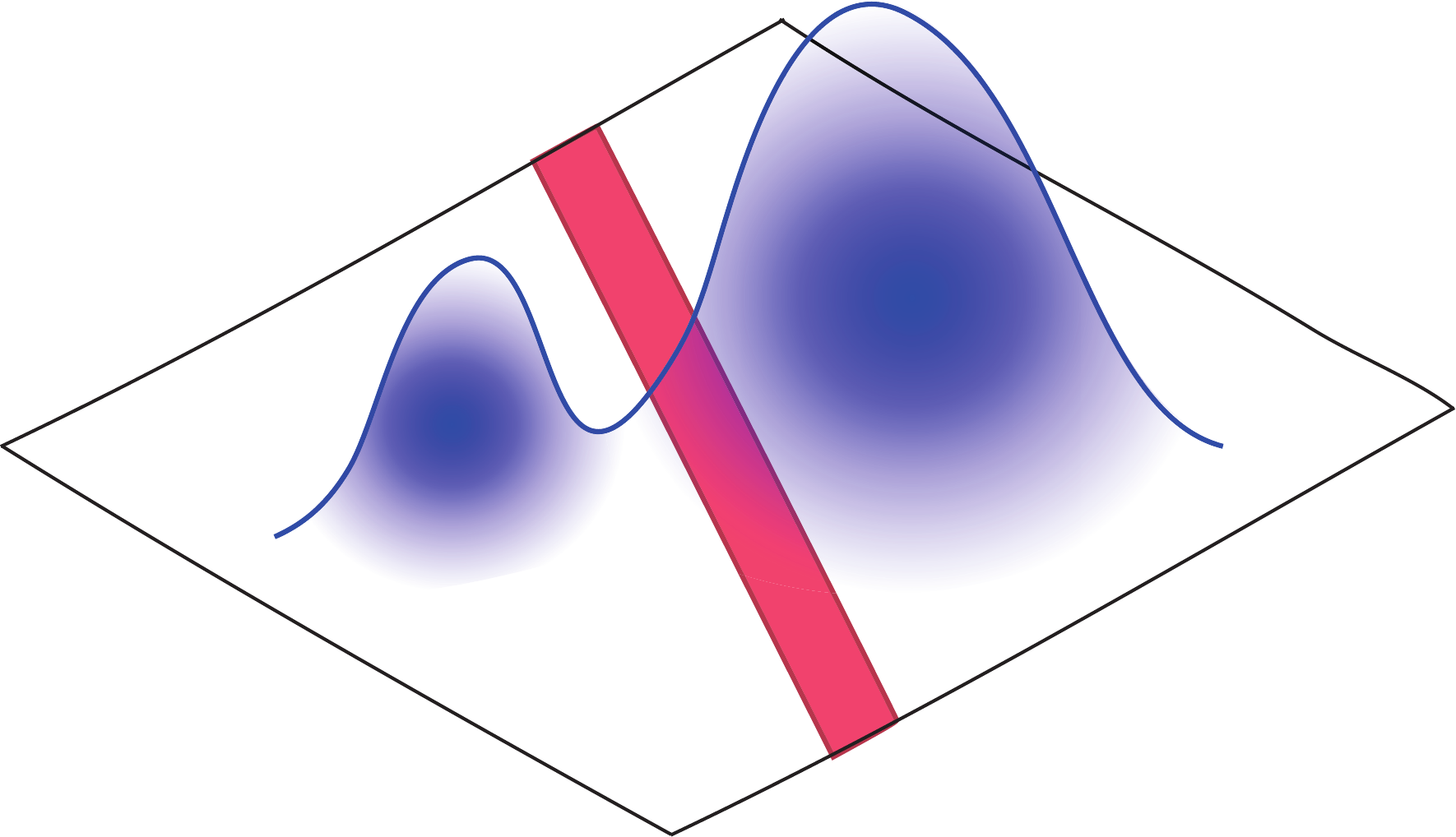}
\put(-45,87){\textcolor{blue}{initial density $p_0(x)$}}
\put(-68,-7){\textcolor{red}{$(Ce^{At})^{-1}(B_y)$}}
\vspace{0.1cm}
\caption{Illustration of the relation between initial state distribution and output distribution at a given time. The value $\mathbb P_{y(t)}(B_y)$ is equal to the strip integrals $\int_{(Ce^{At})^{-1}(B_y)} p_0(x) \, \text{d}x$. Taken from \cite{zeng2015tac}.}
\label{fig:density_tomography}
\end{figure}
 The remaining diffculty is then due to the fact that we only know the integrals over sets that stretch to infinity. Thus, for a single time point, we can never know $p_0$, since certain information about $p_0$ is simply integrated out. 

Thus, we may only hope that as time changes, the directions of the strips, dictated by $Ce^{At}$, change and that the information for different directions can be combined to infer the integrand $p_0$. This is precisely the same problem as in tomography problems, where one wants to obtain a cross-section of an object by taking radiographs from different angles. Our study of the ensemble observability problem established a direct mathematical connection between (ensemble) observability and tomography problems. In fact, the analogy is rather evident on a conceptual level, because both problems are well known to be about inferring internal information from external measurements, which in systems theory typically refer to the internal state and the external output, respectively, and in tomography refer to an internal structure of a body and radiographs, respectively.

Thus, in addition to the original, dynamic viewpoint, there is this second viewpoint associated to the ensemble observability problem in which we do not consider the evolution of the initial state distribution with the flow, but instead, the evolution of the ``measurement directions'', which are dictated by $\ker Ce^{At}$. For the example of the harmonic oscillator, the directions at which we take projections of the initial state distribution, rotate in a uniform counter-clockwise motion, which is in fact the canonical example of a tomography problem, by which the reconstructability of the ensemble observability problem for the harmonic oscillator becomes very clear. Figure~\ref{fig:harmonic_oscillator_duality} illustrates the duality between the two different viewpoints.

\begin{figure}[htp!]
	\centering
	\includegraphics[width=0.43\textwidth]{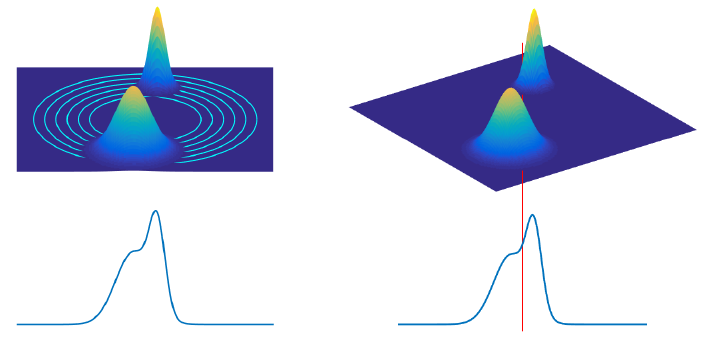}
\vspace{-0.2cm}
	\caption{Left: The distribution evolves with the flow, undergoing a rotation about the origin, and the measurement direction is fixed. Right: The distribution is held fix and we, as a (physical) observer, rotate around the object with our focus fixed on the center of the object. The observed densities are exactly the same in the two different setups.}
	\label{fig:harmonic_oscillator_duality}
\end{figure}

The quite unexpected connection to tomography that was revealed in our investigation of the theoretical problem was effectively leveraged both for theoretical studies, as well as practical reconstruction schemes. In the former, the probabilistic analogue of the projection slice theorem, the Cram\'{e}r-Wold theorem (see Section~\ref{sec:ensemble_estimator}), yielded insightful algebraic geometric conditions for ensemble observability. In the latter, the Algebraic Reconstruction Technique from computed tomography provided a reconstruction method, which, unlike the previous approaches that treated the dynamic aspect more as a black box that is used only for forward simulation purposes, was anchored in a detailed systems theoretic analysis of the underlying problem; comparative studies in light of the new tomography-based viewpoint revealed significant weaknesses of the previous kernel-based reconstruction methods. The curse of dimensionality, however, was also not resolved in this new approach, so that it became apparent that  a purely sample-based approach had to be derived.

To progress towards a sample-based viewpoint, we note that a first observation hinting in this direction can in fact be extracted from the study of an earlier considered nonlinear system, which served as an insightful example for understanding the mechanisms of the ensemble observability problem in the nonlinear case. This is a simple nonlinear oscillator
\begin{align}
\begin{split}
   \dot{x}_1 &= x_2, \\
   \dot{x}_2 &= -4 x_1 + x_1^2, \\
\end{split}
\label{eq:brocket_system}
\end{align}
with output $y=x_1$. As the initial state distribution, we again consider a bimodal distribution, as illustrated on the top left plot in Figure~\ref{fig:nonlinear_ensemble_observability}.
\begin{figure}[htp!]
	\centering \hspace{-0.26cm}
	\includegraphics[width=0.49\textwidth,trim=2.3cm 0cm 2.3cm 0cm,clip]{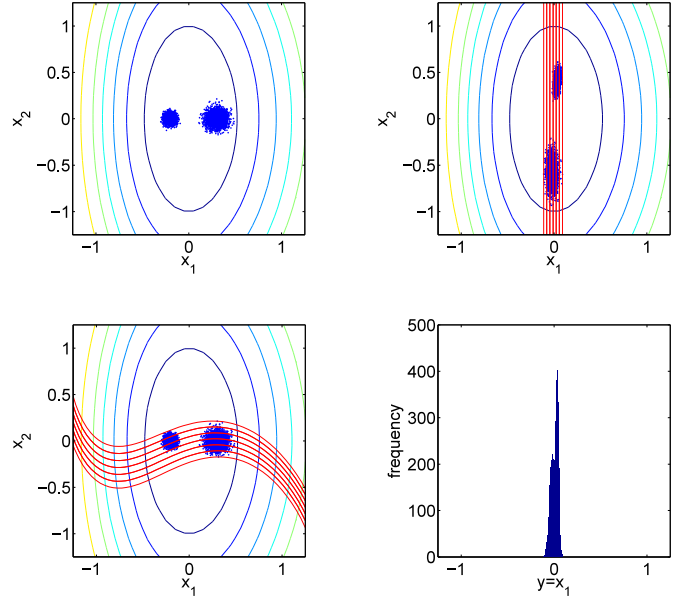}
\caption{Top left: The actual initial state distribution and the phase portrait of the nonlinear oscillator. Top right: The propagated state distribution at a specific time, as well as an indication of the level sets of the output measurement. Bottom right: The histogram corresponding to the measurement of the output distribution at the specific time point associated to the top right figure. Bottom left: The result of running the same system backwards in time, bringing the propagated state distribution to the original initial state distribution and curving the initially straight vertical level sets in the process.}
\label{fig:nonlinear_ensemble_observability}
\end{figure}
The system evolves, and at some time point will be subject to measurement, providing an output snapshot of the ensemble, or, in most practical scenarios, a histogram, as shown in the bottom right plot of Figure~\ref{fig:nonlinear_ensemble_observability}. It is almost a triviality to see that the number of samples in a given bin in the histogram is the same as the number of samples squeezed between the two lines corresponding to the boundary of that bin. Now to relate this measured information to the initial state distribution that we are actually interested in, the idea is to apply the reverse flow to the whole content of the upper right plot in Figure~\ref{fig:nonlinear_ensemble_observability}, i.e.\ the propagated state distribution and the red lines. This will bring back the state distribution to its initial position and -- this perhaps being the more interesting part -- curve the previously straight vertical level sets. 

Now it should be intuitively clear, that throughout the application of the reverse flow, the number of particles between two lines will stay the same; there will be no ``crossing'' in state-space whatsoever, which is being guaranteed by one of the basic properties of a flow. This insight paired with the aforementioned triviality regarding the measured histograms now is the key to sample-based methods.

Before we proceed with discussing the purely sample-based perspective, we note that the continuous limit of these simple observations is in fact a very intuitive way of viewing the measure-preserving property that the flow defined for densities admits. This measure-preservation is also often referred to as a continuity in the physical community and can be written as
\begin{align*}
 \int_{(h \circ \Phi_t)^{-1}(B_y)} p_0(x) \, \text{d}x = \mathbb P_{y(t)}(B_y).
\end{align*}
This is a generalization of the result in the linear case, where only $x \mapsto Ce^{At}x$ being replaced by the nonlinear forward mapping $x \mapsto (h\circ\Phi_t)(x)$; the general continuity principle remains valid in the nonlinear regime. Now to this pushforward equation, one can again associate a tomography problem, a nonlinear one, in which the ``scanning geometry'', i.e. the shape of the curved strips we are integrating over, is determined by nothing other than the interplay between the level sets of the output mapping and the flow of the vector field, which, from a systems theoretic perspective, is one particular aspect of the observability problem that makes it so interesting. As we saw, this very geometric viewpoint on the observability problem is particularly accentuated in light of the framework of ensembles with a distribution in states.

\section{Formulating the ensemble state estimator}
\label{sec:ensemble_estimator}
So far, we have articulated the need to consider a new type of approach in the computational ensemble observability problem, in which the sought state distribution is to be reconstructed by means of finitely many samples of it. The key problem in establishing this is to find suitable update and correction rules for the individual observer states so that the ensemble of observer states eventually converges to a configuration that is very likely to be a set of samples from the unknown distribution.  As discussed in the introduction, this will be done in a manner similar to Markov chain Monte Carlo methods, such as the Metropolis-Hastings algorithm \cite{hastings1970monte}, though in this paper, instead of steadily generating sample points based on proposal and acceptance rules, we start out with a fixed number of particles and perform (randomized) actions on the $N$ particles to arrive at a final configuration that is to approximate the initial state distribution. From a broader point of view, this amounts to a first solution of the statistical analogue of the classical tomography problem.

It turns out that the presented sample-based derivation of the nonlinear pushforward equation in fact already contains all important ingredients to successfully establish a sample-based framework. It is noted however, that the exact implementation is still far from being obvious at this stage and requires further discussions. Essentially, the key idea that will enable our sample-based undertaking is in fact all along encoded in the Cram\'{e}r-Wold theorem \cite{cramer1936some}, which, in one of its different most prominent versions, states that if for two joint distributions \emph{all} marginals distributions \emph{in all directions} are the same, then the joint distributions are the same. Another way to put it is that a joint distribution is uniquely determined by its marginals in all different directions.

\begin{Theorem}[Cram\'{e}r-Wold Theorem]
A distribution of a random vector $X$ in $\mathbb R^n$ is uniquely determined by the family of distributions of $\langle v,X\rangle$, with $v \in \mathbb S^{n-1}$.
\label{cramer_wold}
\end{Theorem}

\begin{proof}
The proof follows from a straightforward computation relating the characteristic function of $\langle v,X\rangle$ with that of $X$,
\begin{align*}
   \varphi_{\langle v, X\rangle} (s) = \mathbb E\big[e^{is\langle v,X\rangle}\big] = \mathbb E\big[e^{i \langle sv, X \rangle}\big] = \varphi_X (sv).
\end{align*}
Since $\varphi_{\langle v, X\rangle} (s)$ is given for all $v \in \mathbb S^{n-1}$ and all $s\in \mathbb R$, by the above identity we know the characteristic function $\varphi_X$ of $X$, and thus also the distribution.
\end{proof}

The Cram\'{e}r-Wold theorem can in fact be easily relaxed to cases in which marginal distributions are not available in all directions, but rather only in a smaller set of directions, which is closely related to the issue of limited angle tomography. In \cite{zeng2015tac}, we studied the underlying mathematical problem and were in particular able to provide complete insight into the connection between the required ``minimal'' set of directions and properties of $(A,C)$, which would, analogous to the classical observability of a linear system, determine whether the underlying system is ensemble observable or not. As we will see, most examples of systems that are ensemble observable will not possess the property that $\ker Ce^{At}$ covers all possible ``directions''. A specific example illustrating this fact very clearly is a double integrator (see Section~\ref{sec:ensemble_observer}).

In light of this particular perspective on the Cram\'{e}r-Wold theorem, the idea would thus be to produce samples in $\mathbb R^n$ so that the projections of the sample points in all available directions are as close as possible to the corresponding output histograms. The key to achieve this is to use an optimal transport approach to measure the closeness between the histograms of the projected samples and the output histograms and to devise a suitable correction strategy that will yield a matching of the two histograms.

Let us discuss this mathematically in the case that the states are $n$-dimensional and that the output is scalar. Let the ensemble state estimator consist of $N$ particles $\widehat{x}^{(i)}$, where $N$ is (of course) taken to be sufficiently large. For each direction $v \in \mathbb R^n$, suppose that we have $M$ particles $\langle v, x^{(i)} \rangle$, where the $x^{(i)}$ are samples from the joint distribution. We then produce a histogram for these measured samples and also produce a histogram for the projected estimator states $\langle v, \widehat{x}^{(i)}\rangle$ with the same bins $[v_j, v_{j+1}]$ with $j = 1, \dots , \ell$. The situation is illustrated in Figure~\ref{fig:3d_plot_hist}. 
\begin{figure}[htp!]
	\centering
\includegraphics[width=0.38\textwidth]{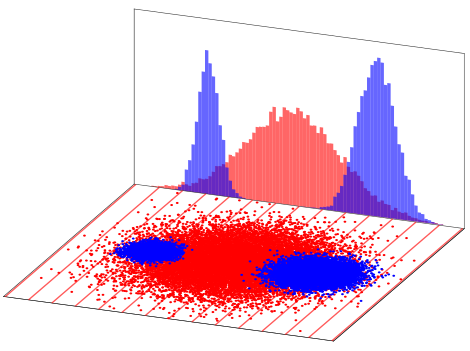}
\vspace{-0.2cm}
	\caption{This figure shows the sample points from the reference distribution (blue) and the estimator's initial configuration of its sample points (red). The histograms of the marginalizations in one particular direction are illustrated in the back. By choosing the same bins for the two histograms, we can describe these as two vectors, whose entries are the (normalized) frequencies.}
	\label{fig:3d_plot_hist}
\end{figure}

When the bins of the two histograms are identical, both histograms can be described by the vectors
\begin{align*}
q^v = \begin{pmatrix} {q}_1^v & \dots &{q}_\ell^v \end{pmatrix}, \;\;\;\;\;\;\;
\widehat{q}^v = \begin{pmatrix} \widehat{q}_1^v & \dots & \widehat{q}_\ell^v \end{pmatrix}
\end{align*}
containing the normalized frequency of projected samples in the respective $\ell$ bins. As such, they are probability vectors, i.e. $\| q^v\| _1 = \|  \widehat{q}^v \| _1 = 1$. The aforementioned correction strategy is then given by ``morphing'' the probability vector $\widehat{q}^v$ into the probability vector $q^v$, i.e.\ to (optimally) redistribute the mass in the different bins of $\widehat{q}^v$ so as to obtain the mass distribution as specified in $q^v$. The problem of transforming one distribution into another by a suitable transport map is illustrated in Figure~\ref{fig:transporting_marginals}.

\begin{figure}[htp!]
	\centering
	\includegraphics[width=0.315\textwidth]{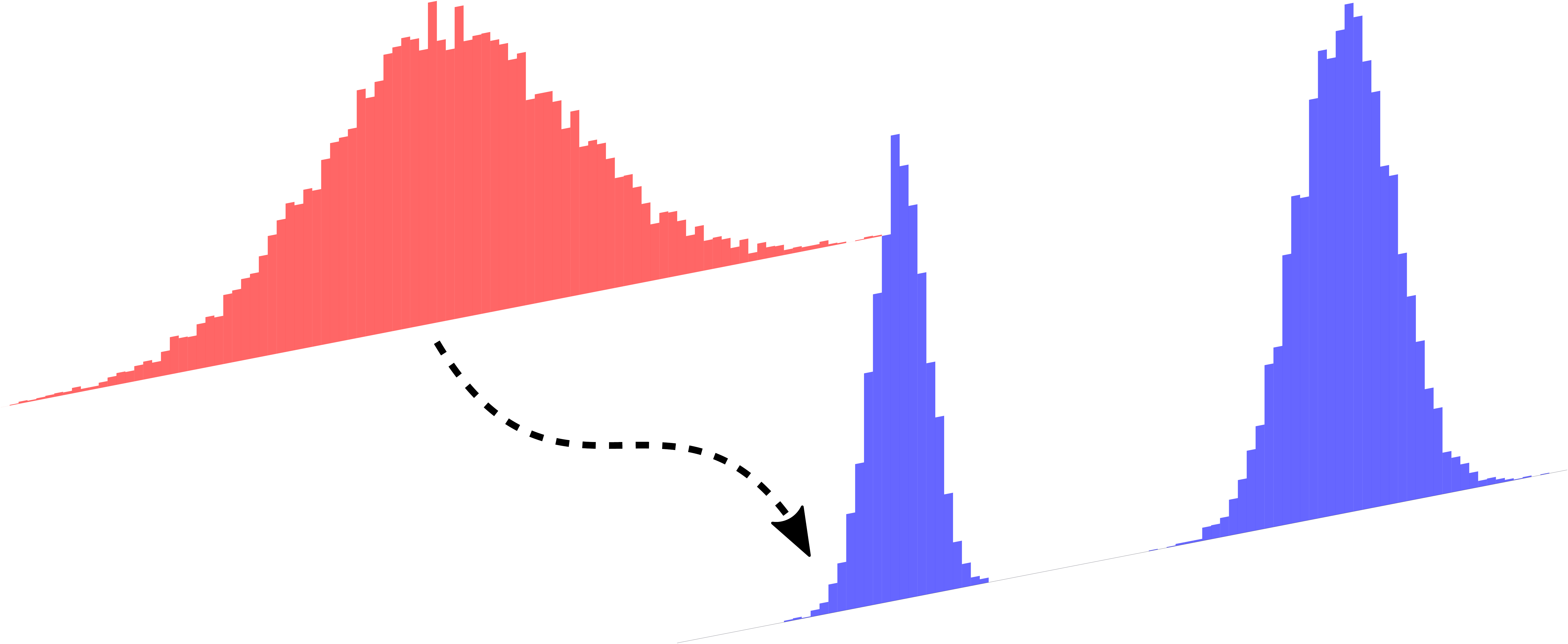}
	\put(-118,11){?}
\vspace{-0.2cm}
	\caption{This figure illustrates the idea of finding a way to transport one distribution into another, or, equivalently, transporting the associated probability vectors into another.}
	\label{fig:transporting_marginals}
\end{figure}

This is in fact the most basic instance of an optimal mass transport problem, namely one in a completely finite-dimensional setting. Here one is seeking for a so-called \emph{transport plan}, which in the discrete setting is specified by a matrix $T\in \mathbb R^{\ell \times \ell}$ with non-negative entries so that
\begin{align*}
  \sum_{i=1}^\ell T_{ij} = \widehat{q}_{j}^v, \hspace{0.5cm}
  \sum_{j=1}^\ell T_{ij} = q_{i}^v.
\end{align*}
The intepretation is that the entry $T_{ij}$ would dictate how much of the (probability) ``mass'' $\widehat{q}_{j}^v$ in the $j$th bin of the histogram is to be transported to the $i$th bin, so that eventually $\widehat{q}^v$ will be completely transformed into $q^v$. 

The aforementioned optimality is incorporated into this framework by additionally considering the cost functional 
\begin{align*}
J = \sum_{i=1}^\ell \sum_{j=1}^\ell |i-j| \, T_{ij}.
\end{align*}
From a physical perspective, this is a very reasonable choice as it favors transport plans that realize the transportation of one mass distribution into another in the most economical way. But this particular choice also leads to additional nice mathematical features, such as the fact that in this case the dual problem is a linear program involving only $\ell$ optimization variables instead of $\ell^2$ variables. This is commonly referred to as the Kantorovich-Rubenstein duality. An even faster way to (approximately) solve this particular case of an optimal mass transport problem for large problem sizes is through the so-called method of Sinkhorn iterations \cite{cuturi2013sinkhorn}. Having solved the optimal transport problem, we obtain the transport plan $T$ for mapping the two vectors containing the frequencies in the different bins, as illustrated in Figure~\ref{fig:transport_plan}.

\begin{figure}[htp!]
\vspace{-0.1cm}
	\centering
	\includegraphics[width=0.4\textwidth]{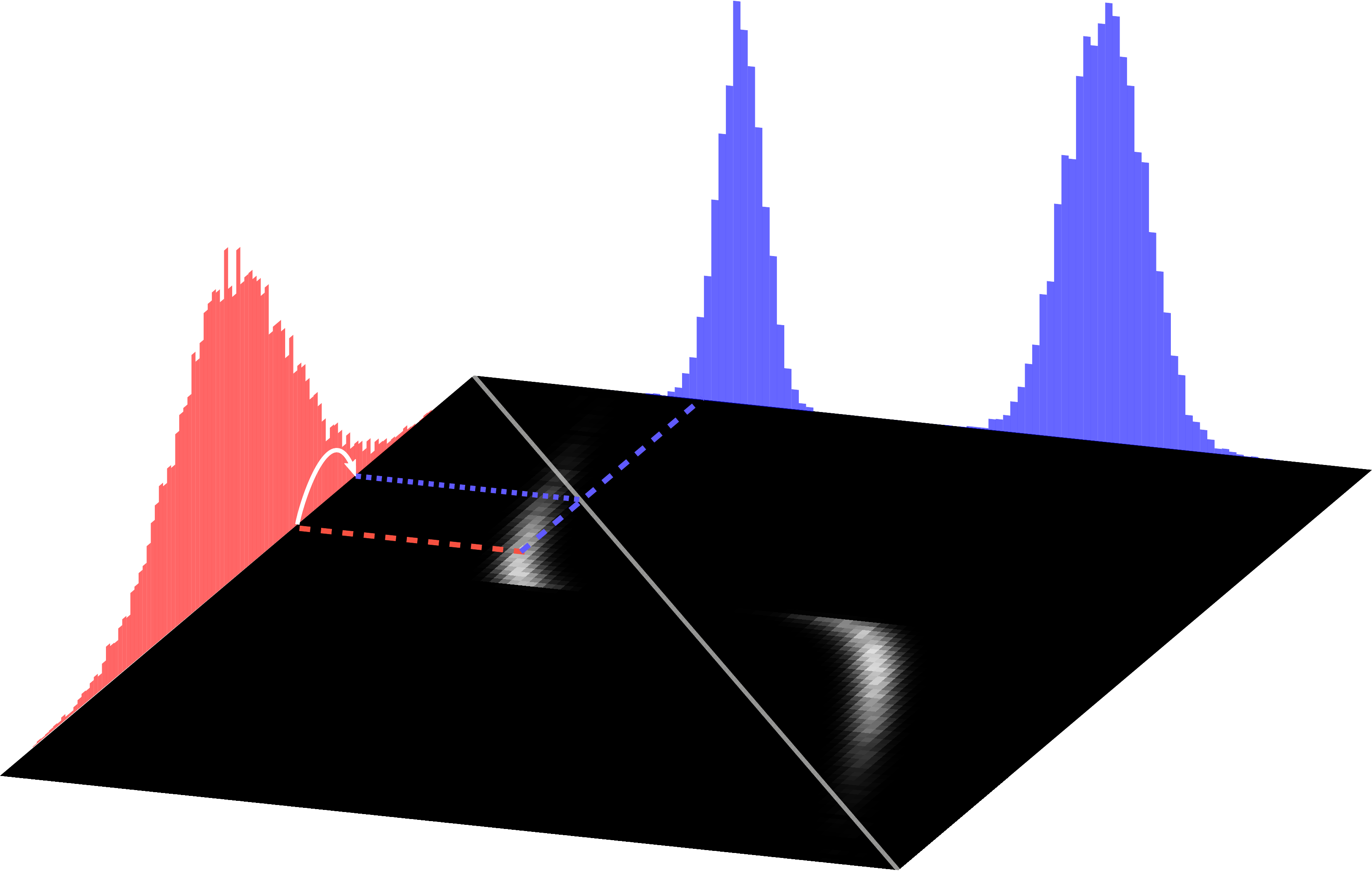}
	\caption{This figure shows a visualization of the transport plan, with the intensity in a pixel corresponding to the magnitude of the corresponding entry in the transport plan matrix (gray scale), as well as the two marginal distributions (red and blue). The red and blue dashed lines indicate how the transport plan is related to the two corresponding marginal distributions. The dotted blue line is the result of reflecting the blue dashed line about the diagonal line, and highlights the position towards which the mass highlighted by the red dashed line is to be transported, as summarized by the white arrow between the two corresponding bins.}
	\label{fig:transport_plan}
\end{figure}

So far, we have discussed a solution that describes which corrective measures have to be implemented on the level of distributions so as to match marginal distributions of the estimator to marginal distributions of the actual particle system. The original problem, however, does not solely consist in solving such an optimal transport problem on the level of vectors, but the vectors result from describing the marginalizations of the original sample points of the original system and the estimator, respectively. Thus, the described optimal transport procedure constitutes only a part of the solution, and to obtain a complete implementation of this correction scheme some further discussion is required. 

In the following, an implementation of this correction scheme \emph{on the level of the original particles} is presented. For all $N$ particles $\widehat{x}^{(i)}$ of the ensemble state estimator find the number $m$ of the bin in which $\langle v,\widehat{x}^{(i)} \rangle$ is contained in. The normalized $m$th line of the transport plan matrix $T$, which is a probability vector, is used as follows: 
With probability $T_{mj}$ the particle $\widehat{x}^{(i)}$ is moved to the $j$th bin, by translating it in the normal direction of $v \in \mathbb R^n$. To ensure a certain ``regularity'' of the resulting set of samples, the exact displacement is also randomized, allowing the corrected particle to lie anywhere in the $j$th bin with equal probability.

Figure~\ref{fig:corrected_two_sides} illustrates a situation, in which the estimator state has been corrected with respect to the highlighted direction, but admits a large deviation with respect to a different direction. Clearly, the above described plan will have to be repeated for sufficiently many directions $v \in \mathbb R^n$. The iteration over all different directions $v$ itself can be iterated several times, similarly to the procedure in the Algebraic Reconstruction Technique in computed tomography.
\begin{figure}[htp!]
\centering
\includegraphics[width=0.38\textwidth]{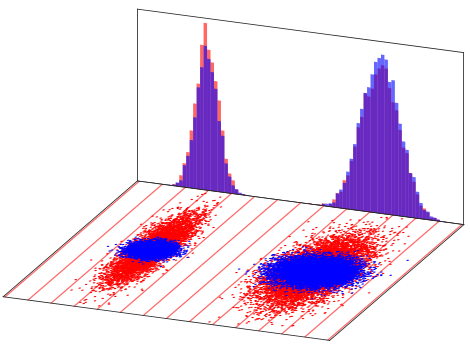}
\vspace{-0.2cm}
	\caption{This figure illustrates the situation in which the presented correction scheme has been carried out with respect to the highlighted direction. The illustrated marginal distribution of the estimator particles matches the marginal distribution of the particles from the actual initial state distribution. Note that the marginal distributions in other directions, e.g.\ that orthogonal to the highlighted one, are clearly not matched, which will eventually have to be addressed in further iteration steps.}
	\label{fig:corrected_two_sides}
\end{figure}
The intuitive idea is that by doing so, we expect to eventually end up with a configuration of particles $\widehat{x}^{(i)}$ whose projections along all given directions are at once in accordance with the actual data. By virtue of the Cram\'{e}r-Wold theorem, in the idealized case that $N \to \infty$, and that all (a sufficient set of) directions are available, we would end up witha perfect approximation of the joint distribution by means of samples of the distribution.

Figure~\ref{fig:harmonic_oscillator_initial_distribution} illustrates the correction scheme for the linear harmonic oscillator, where two correction steps are highlighted. In this particular case, with only two simple iterations, we are already able to achieve a quite acceptable reconstruction.
\begin{figure}[htp!]
	\centering
	\includegraphics[width=0.49\textwidth]{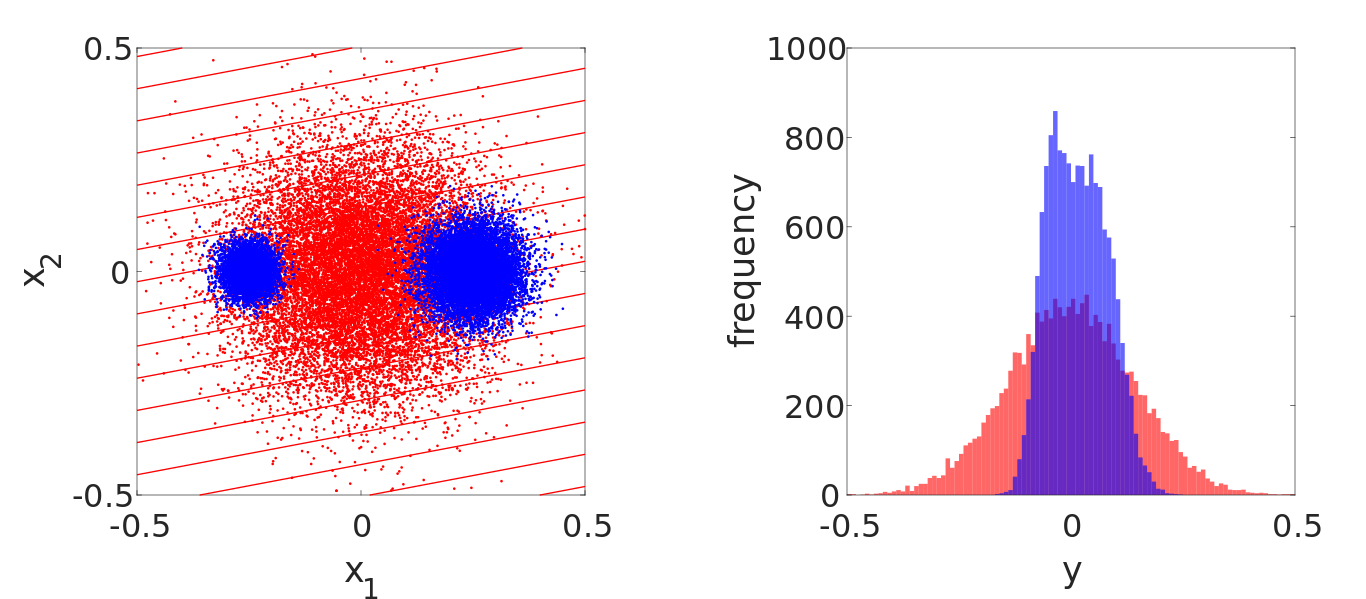} \\[0.3cm]
	\includegraphics[width=0.49\textwidth]{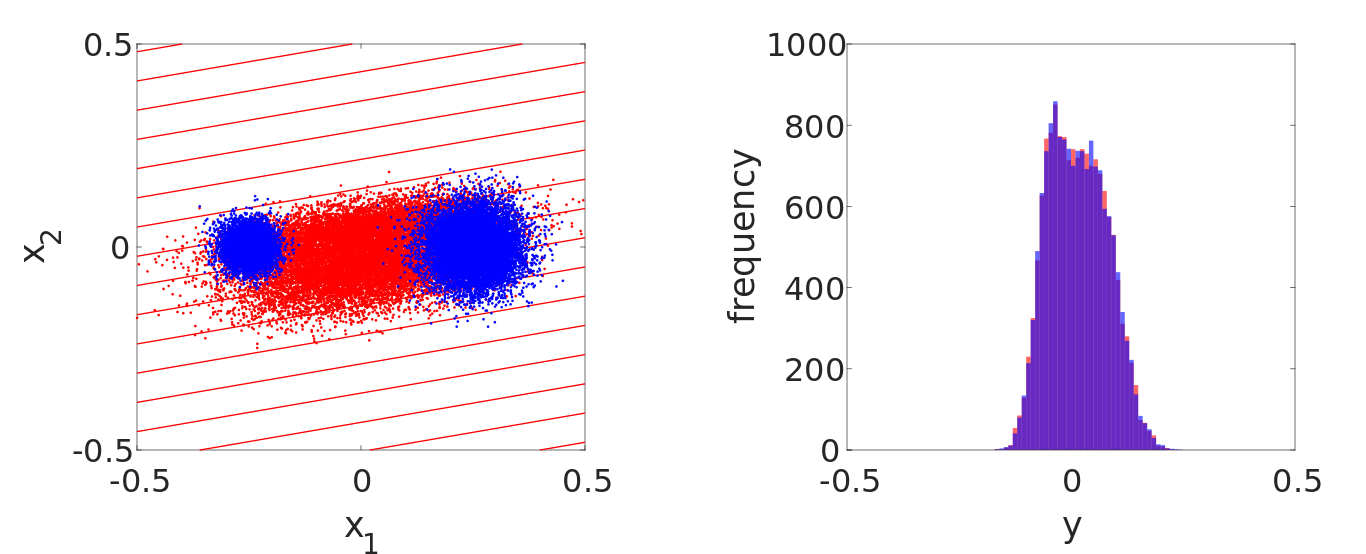} \\[0.3cm]
		\includegraphics[width=0.49\textwidth]{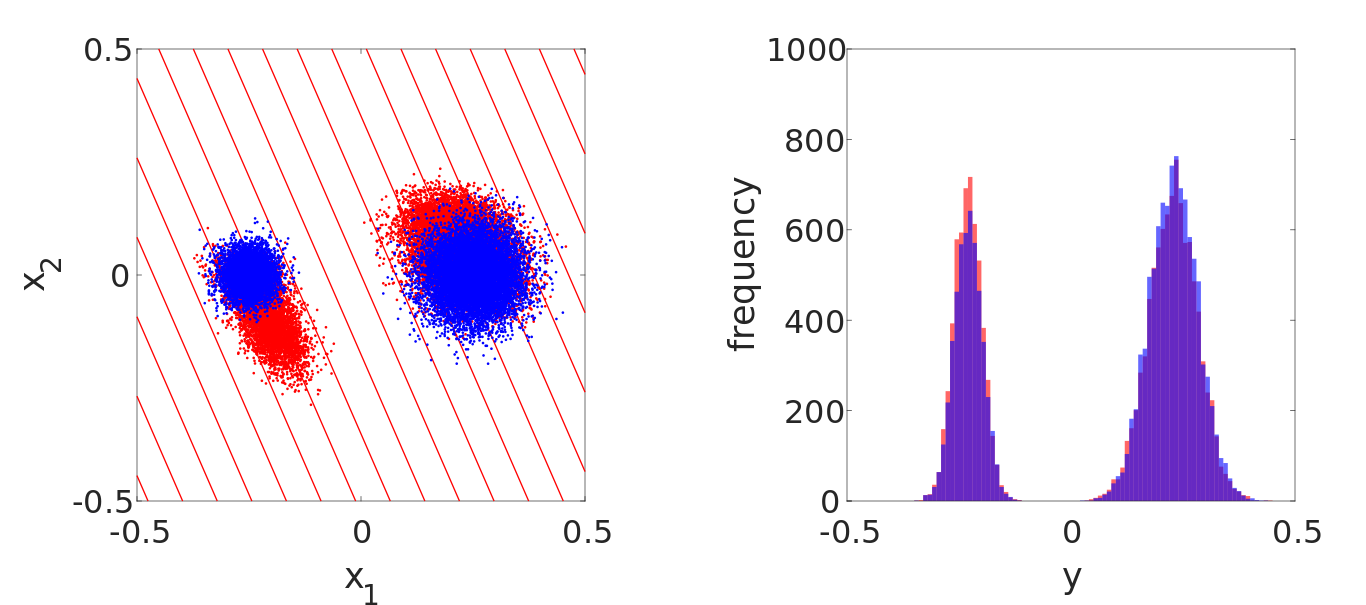} 
	\caption{First row: The actual initial distribution (blue) and a prior estimate (red) are illustrated on the left. The right plot shows the histogram of the projections of the two distributions along the highlighted direction in the left plot. Middle row: The ensemble estimator's state is updated so that the marginals of the projections in the highlighted direction match. Last row: Illustration of a second update of the ensemble estimator's state associated to a different direction.}
	\label{fig:harmonic_oscillator_initial_distribution}
\end{figure}

At this point, we would like to draw some attention to the particular architecture of this correction-based (particle) state estimator. The correction is essentially implemented by means of a two-layer feedback: First, the mismatch between the outputs of the estimator and of the actual system is evaluated on the \emph{population level}, from which a correction on the population level is computed. In particular, at this stage, no attention is paid to individual systems but only the totality of individual systems. The correction  in the next step on the other hand has to be actually realized by implementing it on the level of the individual particles. In particular, it cannot be fully implemented on the population level, i.e.\ by completely broadcasting an instruction to the systems in the ensemble. Rather, different individual systems in the population will be required to receive different instructions (in this case it is based on the bins they are located in). To summarize, though our presented scheme does not operate entirely on a population-level, it is also not a completely individual feedback. Rather it constitutes a quite simple to implement, yet very powerful \emph{hybrid}, given by a two-layer structure, which we may refer to as a \emph{population-level feedback}.

In the case of nonlinear systems, the displacement of the particles for the correction step would need to take place in the direction orthogonal to the curved strips. This would require the computation of the normal direction at each point of the curved level surface, i.e.\ the gradient $\nabla (h \circ \Phi_t)$, which imposes new computational burdens. However, at this point we can again leverage our insight about the conservation of number of samples between two lines, leading us to the idea to correct the propagated particles at time $t$ in the normal direction of $h$ and to apply the reverse flow to the corrected particles, as was illustrated in Figure~\ref{fig:nonlinear_ensemble_observability}. Thus, a simple remedy by which the computation of gradients is circumvented is given by splitting up what in the linear case can be naturally implemented in a single step into two steps by means of an intermediate step. 

To summarize, in the aforementioned unfolded correction procedure, we transport the state distribution of the estimator forward to a given measurement time, compare the its output distribution with the measured output distribution at hand, and then implement the correction at that given measurement time. Then, after the transport plan has been implemented, the state distribution is transported backwards to the initial time. It is to be stressed, however, that the resulting action of this approach on the estimate of the initial state distribution is not necessarily one where the particles were projected orthogonally to the level sets of $(h\circ\Phi_t)$, as the flow $\Phi_t$ need not be angle preserving in general. A detailed illustration of one correction step in the nonlinear case is shown in Figure~\ref{fig:nonlinear_oscillator}, where the resulting action of the unfolded correction procedure is also clearly displayed. The same strategy of course also applies to the linear setting, where it is, however, easier to apply the correction in one simple step.

\begin{figure}[htp!]
	\centering
\vspace{-0.1cm}
\includegraphics[width=0.235\textwidth, trim= 0cm 0cm 1cm  0cm, clip]{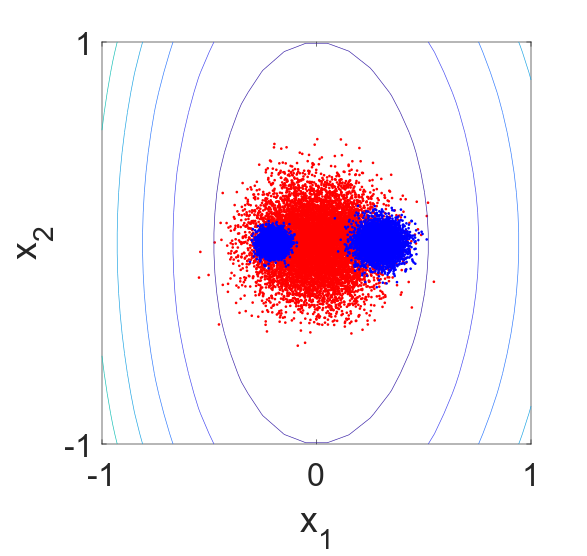} \; \includegraphics[width=0.235\textwidth, trim= 0cm 0cm 1cm  0cm, clip]{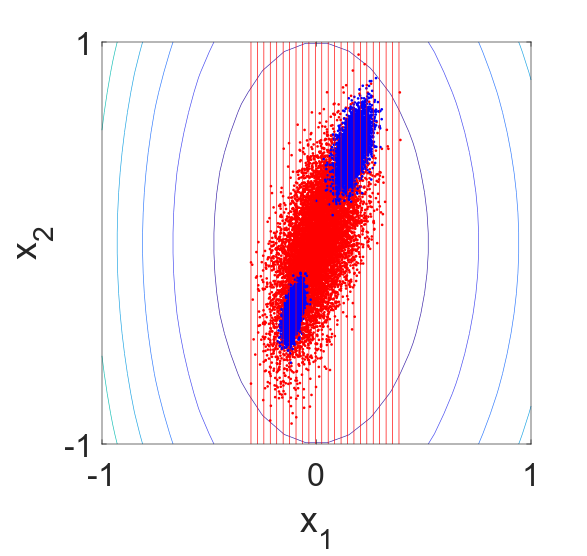} \\ \includegraphics[width=0.235\textwidth, trim= 0cm 0cm 1cm  0cm, clip]{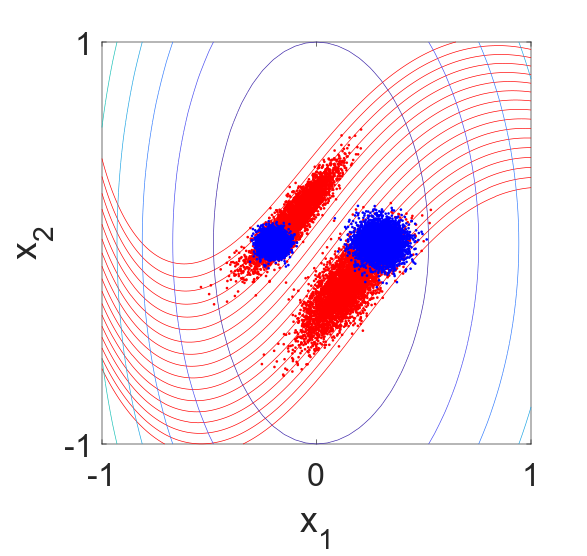}  \; \includegraphics[width=0.235\textwidth, trim= 0cm 0cm 1cm  0cm, clip]{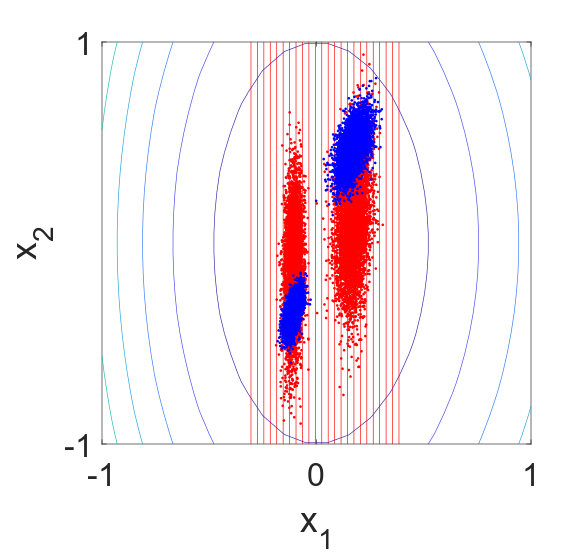}
\vspace{-0.4cm}
\caption{Top left: The initial state distribution (blue) and the estimated initial state distribution (red) before any correction step has been applied. Top right: The two distributions after being transported with the nonlinear oscillator to a given time point, as well as the level sets of the output measurement. Lower right: Correction step using optimal mass transport. Lower left: The transported corrected distribution, as well as the transported level sets.}
\label{fig:nonlinear_oscillator}
\end{figure}

\section{The ensemble observer}
\label{sec:ensemble_observer}
In the previous section, we presented a novel particle-based approach for estimating the initial state distribution of an ensemble from output samples. As for any such state estimation problem, we assumed to have all the measurements at different times stored and available to us at once. Another type of state reconstruction scheme is in a more dynamic spirit, in which the system's state is to be estimated online, i.e.\ at each time instant, the estimated state is updated based on the measurement received at that time point, or, more generally, from past measurements received up to that time point. From a more mathematical point of view, the problem considered in this section is the estimation of $p_{x(t)}$ from past output measurements $p_{y(\tau)}$, with $\tau \le t$, which, when formulated in these more theoretical terms, we recognize to be analogous to a classical filtering problem. So far, approaches to implement such a filtering approach have not yielded any fruits.

To illustrate particular difficulties that were encountered in the aforementioned approaches, we shall highlight two natural approaches that one would rather naturally consider in this context. The first approach would consider a partial differential equation describing the evolution of the estimated state distribution. It is well-known that the original ensemble system can be described by a linear partial differential equation, the Liouville equation \cite{brockett2012notes}, given by
\begin{align*}
 \frac{\partial}{\partial t} p(t,x) = -\text{div}(p(t,x)f(x)),
\end{align*}
where $p(t,\cdot)$ denotes the state density at time $t$. The output distribution results from the state distribution by a marginalization along $\ker C$, i.e.
\begin{align*}
 p_{y(t)}(y) = \int_{Cx=y}  p(t,x) \, \text{d}S.
\end{align*}
We denote the mapping $p(t, \cdot) \mapsto p_{y(t)}$ by $\mathcal C$. In the spirit of the classical Luenberger observer \cite{luenberger1971introduction}, having one part simulating the system and another part correcting based on the incoming output measurements as its basic design principle, it is indeed natural to consider an observer described by 
\begin{align*}
 \frac{\partial}{\partial t} \hat{p}(t,x) = -\text{div}(\hat{p}(t,x)f(x)) + \mathcal L[ \hat{p}_{y(t)} -  p_{y(t)}],
\end{align*}
where $\hat{p}_{y(t)} = \mathcal C \hat{p}(t, \cdot)$. Defining 
$
e(t,x) := \hat{p}(t,x) - p(t,x)
$
as the estimation error, in the approach based on partial differential equations, the problem boils down to designing the (linear) operator $\mathcal L$ so that the error dynamics
\begin{align*}
    \frac{\partial}{\partial t} e(t,x) = -\text{div}(e(t,x) f(x)) + (\mathcal L\mathcal C e)(t,x)
\end{align*}
is asymptotically stable. However, due to the fact that the action of $\mathcal C$ is a rather unique one, not falling into any well-studied category of operators in the theory of infinite-dimensional systems theory \cite{curtain1995introduction}, a general solution to this stabilization problem remains out of reach.

Another natural idea that circumvents the infinite-dimensional setting is to first discretize the state space, e.g.\ by approximating the considered probability density functions by piecewise constant functions, and then to reformulate the system dynamics for these finite-dimensional approximations. However, in trying to do so, we will at some point encounter a rather fundamental problem associated to this idea, which can be already seen for a simple linear oscillator. If the discretization of the state space is not tailored to the specific vector field at hand, say, we choose a simple discretization into pixels in $\mathbb R^2$, then the resulting discretized linear system will no longer admit the mass preserving property. This is because in implementing this discretization scheme, we inevitably have to truncate the discretization of the state space to some region of interest, of which the boundaries will suffer from leakage of mass, but will not provide mass from outside, the outside part being truncated. Thus, for an observer based on this idea of discretization, the part that simulates the system will not be able to reproduce the actual system behavior. In fact, the state generated by the simulation part will naturally converge to zero as the incoming flow inevitably has to be truncated, and the general trend will thus be that the whole mass will eventually leak out at the boundaries.

Using the new insights from our first sample-based implementation, we are already able to formulate a new sample-based ensemble observer, at least for the case of the harmonic oscillator. There, we just let the system evolve and keep correcting the mismatch between the output distributions instantaneously. Due to the duality illustrated in Figure~\ref{fig:harmonic_oscillator_duality}, this results in essentially the same correction scheme as in the state estimation case, which makes the results from the foregoing state estimation case directly applicable to the ``dynamic estimation'' of a harmonic oscillator. The application of such a strategy is, however, not always feasible, as will be discussed and highlighted in the next subsection for the example of a double integrator. There we will also establish an observer for arbitrary ensemble observable systems based on a moving horizon estimation scheme that batches past measurements and processes these along the lines of the estimation of the initial state distribution.

\subsection{Moving horizon ensemble estimator}
An example of a system in which a one-step (memoryless) approach does not yield satisfactory results is given by
\begin{align*}
 \dot{x} &= \begin{pmatrix} 0 & 1 \\ 0 & 0 \end{pmatrix} x, \\
	y &= \begin{pmatrix} 1 & 0 \end{pmatrix} x,
\end{align*}
which is a simple double integrator. We can directly compute $Ce^{At} = \begin{pmatrix} 1 & t \end{pmatrix},$ allowing us for the specific example of a double integrator to write down the relation between angle $\alpha$ of $\ker Ce^{At}$ and time $t$ explicitly as
\begin{align}
\tan(\alpha) = \frac{x_2(t)}{x_1(t)} = t   \;\; \Leftrightarrow \;\; \alpha = \arctan(t).
\label{eq:arctan}
\end{align}
This very simple reading relation shows that unlike in the example of a harmonic oscillator, the maximal spread of achievable angles is inherently restricted to the range of $t \mapsto \arctan(t)$. Moreover, the explicit relation allows us to choose the time points of measurement $t_k$ in such a way that the corresponding set of angles is uniformly distributed, which in turn is expected to yield better results for the reconstruction.

Due to the lack of a persistent, or, recurrent, oscillation encoded in the mapping $t \mapsto Ce^{At}$ in the case of a double integrator, the foregoing simplistic strategy of an instantaneous (memoryless) correction will not be applicable. Recalling the filtering formulation introduced in the beginning of this section, where the problem is to estimate $p_{x(t)}$ from past output measurements $p_{y(\tau)}$, with $\tau \le t$, we note that unlike in the classical linear setup with a single point particle, at this point it does not seem to be possible to get a solution as elegant and \emph{fully recursive} as the Kalman filter for the ensemble case. One major cause is that due to computational and also memory restrictions, further restrictions on the horizon, which in the purely theoretical framework would be specified by $\tau \le t$, need to be placed. In view of a more practical formulation, a more realistic choice would be restricting the time points at which output data is available to
$
 t-T_H \le \tau \le t,
$
where $T_H$ denotes the horizon length and $[t-T_H,t]$ is called the moving horizon. Of course, when practically implementing such a moving horizon scheme, we also need to further assume that the available measurement times are not continuous, but discrete time points. 

We will show in the following that the estimation of the current ensemble state distribution from past output distributions (in the measurement horizon) is inherently dual to the estimation of the ensemble initial state distribution from the generated output distributions (forward in time). For a single particle, the relation between the state $x(t)$ and the output $y(\tau)$ at an earlier time point $\tau \le t$ is given by
\begin{align*}
Ce^{A(\tau-t)}x(t) = y(\tau).
\end{align*}
On the level of the distributions, this would translate to the description that the state distribution of an ensemble $\mathbb P_{x(t)}$ is related to the output distribution $\mathbb P_{y(\tau)}$ at an earlier time point $\tau \le t$ through a pushforward relation where the forward mapping is the linear map $x \mapsto Ce^{A(\tau-t)}x$. To implement this, we apply a procedure dual to the one for estimating the initial state distribution. The state distribution $\mathbb P_{\hat{x}(t)}$ that we would like to estimate at time $t$ will be constantly corrected based on previous measurements associated to time points $\tau \le t$  within the horizon by the same procedure as in the estimation problem for the initial state distribution: We propagate the particles of $\mathbb P_{\hat{x}(t)}$ with $x \mapsto Ce^{A(t-\tau)}x$ to the output distribution $\mathbb P_{\hat{y}(\tau)}$, compute a correction strategy based on the mismatch between $\mathbb P_{\hat{y}(\tau)}$ and $\mathbb P_{y(\tau)}$, and implement the correction on the particles of $\mathbb P_{\hat{x}(t)}$. Figure~\ref{fig:MHE_double_integrators} illustrates a resulting tracking process using the proposed method for an ensemble of double integrators.
\begin{figure}[htp!]
	\centering
	\includegraphics[width=0.45\textwidth]{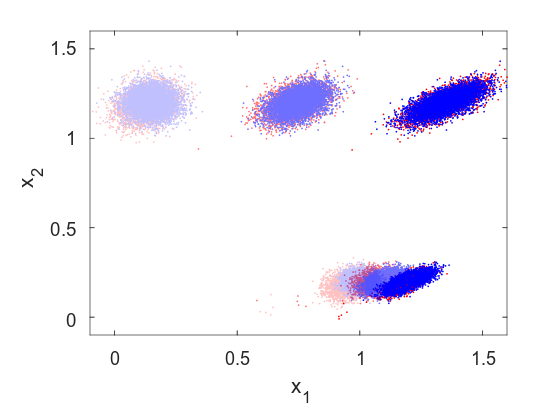}
\vspace{-0.2cm}
	\caption{Three successive predictions at $t = 0.5, 1.0, 1.5$ (distinguished by transparency), each of which is computed from 10 time points in the measurement horizon with $T_H = 3$. }
	\label{fig:MHE_double_integrators}
\end{figure}

In this example, the specific horizon length $T_H = 3$ was chosen so as to guarantee a sufficiently large spread of available directions for each correction step. Note that by defining $\tau':= \tau - t$, which takes values in the interval $[-T_H,0]$, we see that the directions are dictated by $\ker Ce^{A \tau'}$, where $\tau' \in [-T_H,0]$, which result from transporting $\ker C$ forward in time in the interval $[0,T_H]$. Due to the simple relation between angles and times established in \eqref{eq:arctan}, the range of available angles would be 0 to $\text{arctan}(3) \approx 71.56^{\circ}$. In order to facilitate a wider spread, a longer horizon would need to be provided. For example, in order to have a spread of $85^{\circ}$, it would already require a horizon length of $T_H = 11.43$. 

Within the measurement horizon, the output distributions of the actual ensemble are of course not measured continuously, but at discrete time points. In the example, the times at which measurements of the output distributions are available are
$
 t_k  = t - 0.1k,
$
where $k = 1, \dots, 30$.
Out of these 30 measurement times only 10 measurement times are actually utilized by incorporating the measurement data for the correction steps at each prediction step. Of course, one could in fact increase the number of time points used for the reconstruction at each time step, and also increase the number of correction steps performed at each prediction step. This would, however, result in an increased computational load at each prediction step. 

We note that in the implementation of this example, the 10 time points are chosen randomly from the above measurement times in such a way that the distribution of corresponding angles would be as uniform as possible. More specifically, due to the nonlinear relation $\alpha = \tan(t)$, choosing random times $t_k$ from a uniform distribution defined over $\{t_k\}$ would not result in a uniform distribution of the corresponding $\alpha_k$. Instead, one has to sample with respect to a specific (discrete) distribution $t_k \sim P_t$, which guarantees that the distribution for $\alpha_k = \tan(t_k)$ is (close to) a uniform distribution. The detailed discussion of these issues, while of great practical importance, is beyond of the scope of this paper.

In summary, with the above described procedure we obtain a quite satisfactory method for solving the continuous ensemble observability problem in an on-line fashion. In the following section, we will turn to the study of the closely related discrete ensemble observability problem. We will also see that for single-output systems, we can derive another novel, yet very natural particle tracking method from the considerations of the discrete case. This second method is based on an even more simplistic formulation, but at the same time computationally more demanding as it scales directly with the number of systems $N$, whereas the method presented in this section is designed for systems with large $N$, and, within this regime, scales only with the number of bins of the histograms only, which is much more favorable for very large $N$.

\section{The discrete ensemble observability problem}
\label{sec:discrete_ensemble}
Our proposed solution for the observer design of continuous ensembles in the foregoing sections also yields direct implications for the discrete version of the ensemble observability problem \cite{zeng2017tac}. In this closely related, but significantly different discrete setup, we consider a fixed number of $N$ systems and at each time step, the $N$ corresponding outputs are measured, however, in an anonymized fashion, i.e. the set of recorded output measurements of the $N$ different systems lack any information relating an individual output measurement in the set to the corresponding one system that yielded the measurement. This is also known as the multitarget tracking problem \cite{bar1978tracking}, where the rather unique premise is related to the type of measurement devices typically utilized in the domains related to multitarget tracking. 

This particular premise makes state estimation for multiple targets highly nontrivial, and the field of multitarget tracking has been subject to an extensive study \cite{smith1975branching, bar1978tracking, leven2009unscented}. Previous work was mainly aimed at developing approaches for a practical solution for multitarget tracking. On the other hand, the multitarget tracking problem is a very fundamental problem offering plenty of theoretical questions and challenges which have not been fully explored systematically. In \cite{zeng2017tac} we aimed to address the multitarget tracking problem from a more conceptual and theoretical point of view. In anticipation of a unification of the discrete and the continuous frameworks developed in \cite{zeng2015tac}, in \cite{zeng2017tac}, the problem was already formulated using the framework of discrete measures. This formulation will now indeed serve as a bridge by which different insights about the computational problem from the continuous case can be immediately applied to the discrete case as well. 

By virtue of this formulation, it is now also trivial to apply the methodology of optimal mass transport to the discrete case. The optimal transport problem of discrete measures is in fact a well-known special case (also known as an assignment problem), where the goal is to transport one discrete measure $\mu_y = \sum_{i=1}^N \delta_{y^{(i)}}$ to another discrete measure $\mu_{\widehat{y}} = \sum_{i=1}^N \delta_{\widehat{y}^{(i)}}$, i.e. associating $y^{(i)}$ to $y^{(\sigma(i))}$ with a permutation $\sigma$ in such a way that the cost functional
\begin{align*}
 J = \sum_{i=1}^N \| y^{(i)} - \widehat{y}^{(\sigma(i))} \|
\end{align*}
is minimized. When the outputs are scalar, in which case the discrete measures are defined on the real line, the solution to the assignment problem turns out to be particularly simple. Here one first sorts the randomly ordered tuples $(y^{(1)}, \dots, y^{(N)})$ and $(\widehat{y}^{(1)}, \dots, \widehat{y}^{(N)})$ in an increasing order. The corresponding permutations that realize this sorting are labeled $\sigma$ for the tuple $y$ and $\widehat{\sigma}$ for the tuple $\widehat{y}$. One can rather easily convince oneselves that the optimal assignment is then given by the permutation $\sigma^{\star} := \widehat{\sigma}^{-1} \circ \sigma,$ i.e.\ by pairing $y^{(i)}$ with $\widehat{y}^{(\sigma^{\star}(i))}$, which corresponds to designating $y^{(i)}$ to be transported to $\widehat{y}^{(\sigma^{\star}(i))}$. Given these pairings, the correction is implemented on the state space by projection in the normal direction, or, in other words, by orthogonally projecting the particle $x^{(i)}$ to its assigned hyperplane, which is defined by $$Ce^{A t_k} x = y^{(\sigma^{\star}(i))}(t_k).$$

As in the continuous ensemble observability problem, the insights gained from the connection to optimal mass transport problems can be directly leveraged to provide a solution to both the problem of estimating the initial state distribution from output data recorded over a given time frame, as well as the online observation problem. The above described procedure is illustrated on a small-scale example in Figure~\ref{fig:initial_state_estimation_discrete}, where an ensemble of five double integrators is considered. 
\begin{figure*}[htp!]
	\centering
  \includegraphics[width=0.315\textwidth]{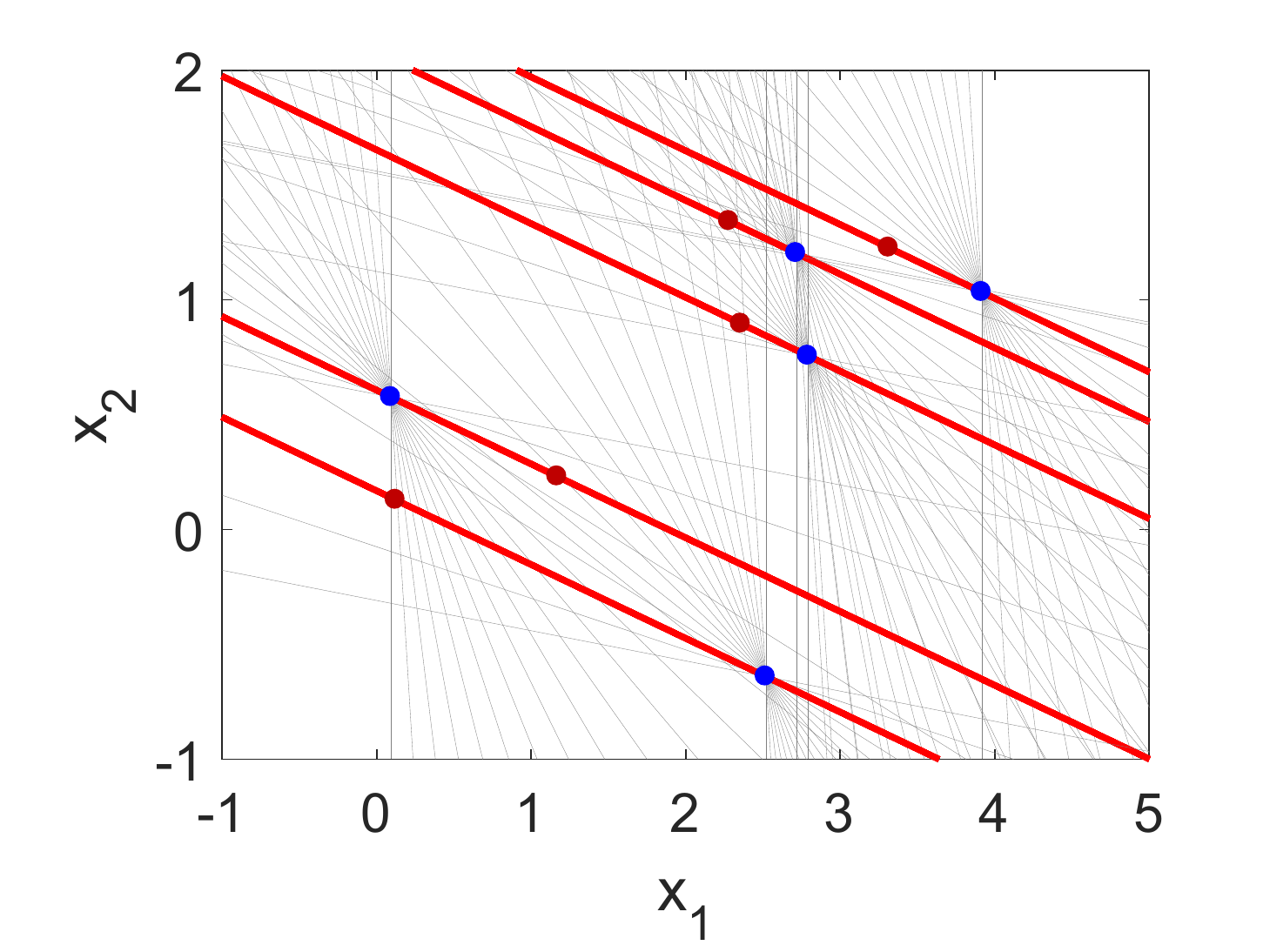} \, \includegraphics[width=0.315\textwidth]{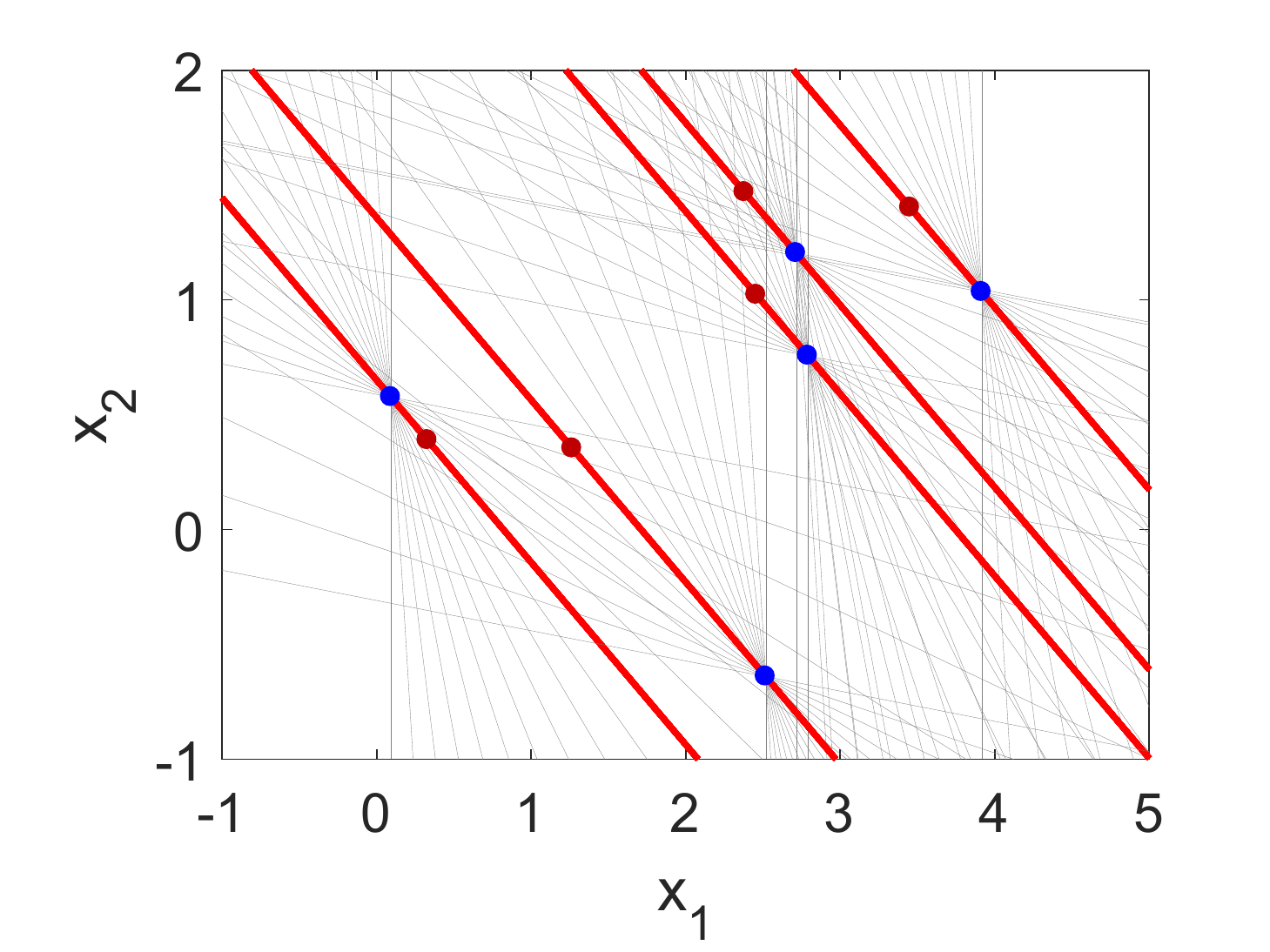}  \,  \includegraphics[width=0.315\textwidth]{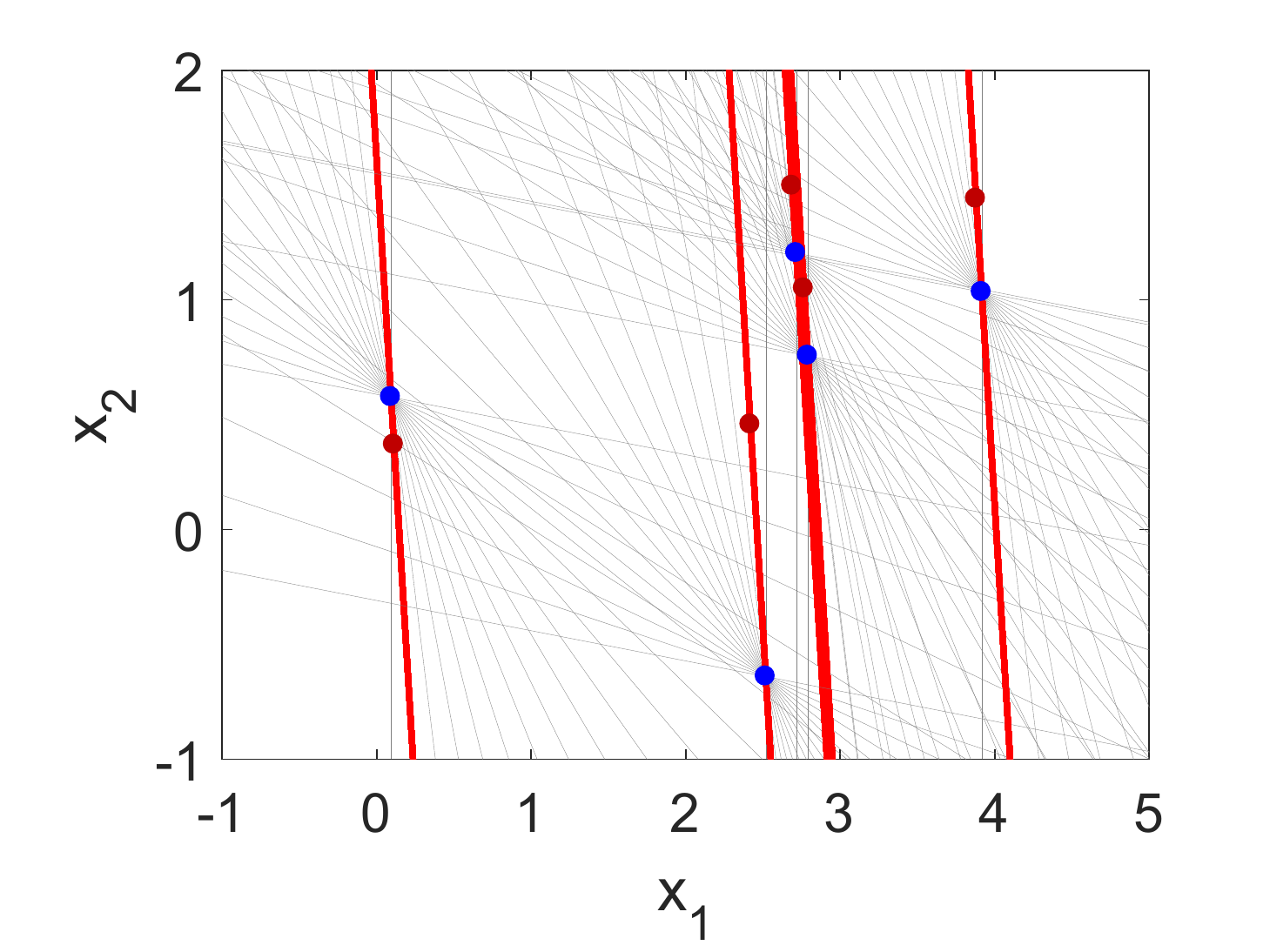}
  \includegraphics[width=0.315\textwidth]{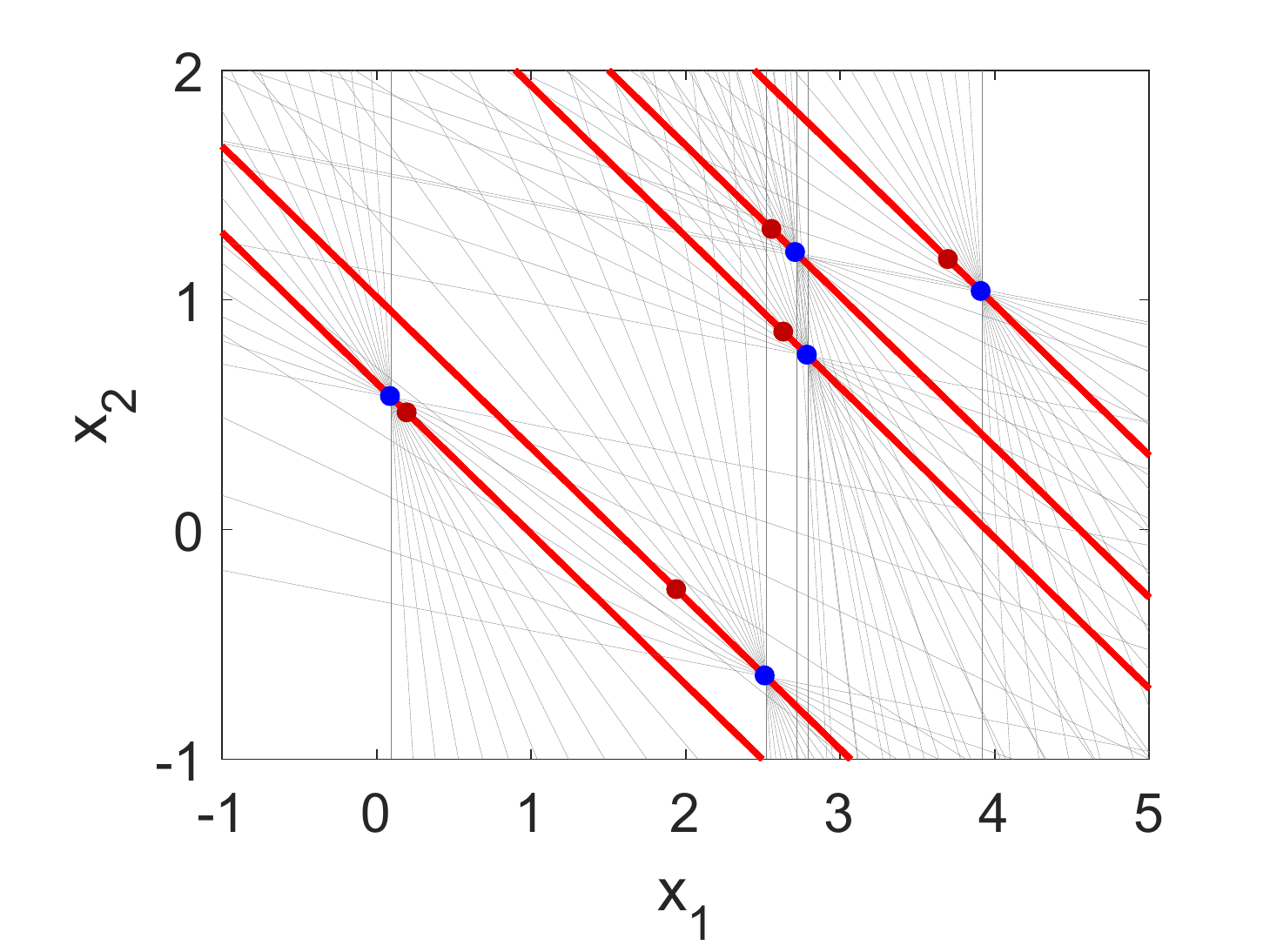} \,  \includegraphics[width=0.315\textwidth]{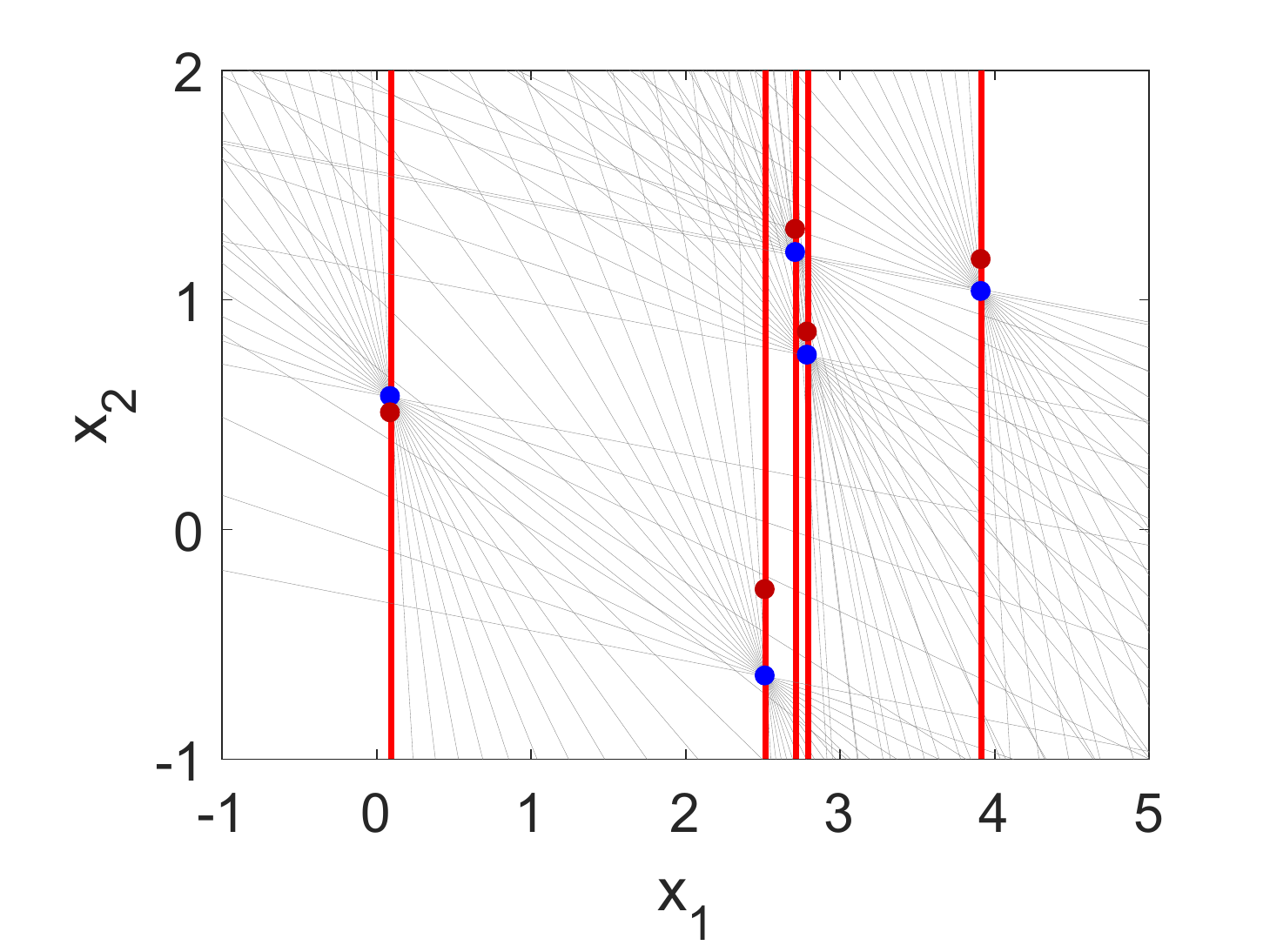} \,   \includegraphics[width=0.315\textwidth]{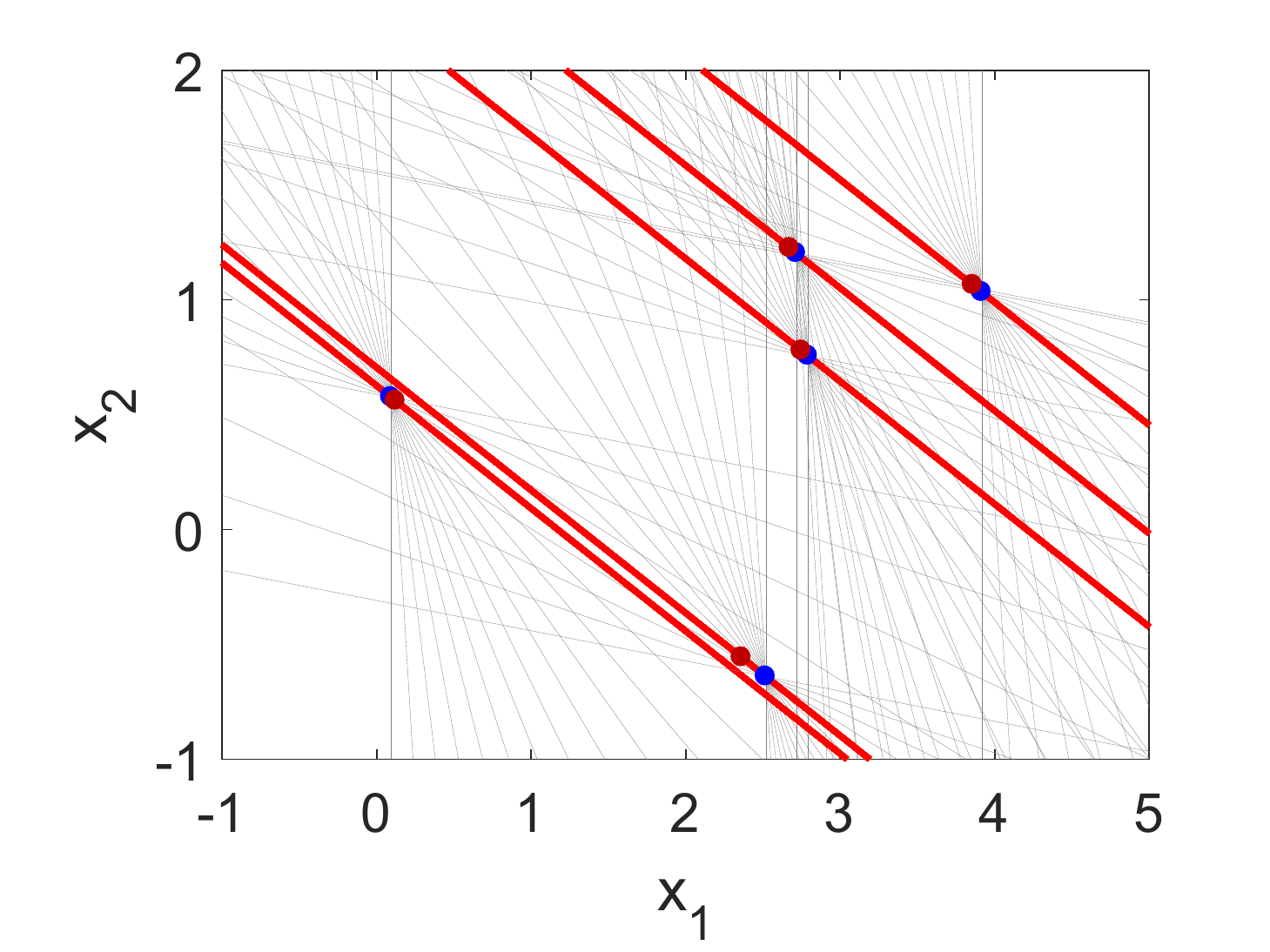}
\vspace{-0.1cm}
	\caption{Illustration of the correction principle for discrete ensembles based on a small-scale example. The time points used for the estimation are chosen so that the corresponding directions are uniformly spread, evident from the illustrated grey ``backprojection lines''.}
	\label{fig:initial_state_estimation_discrete}
\end{figure*}

Through iteratively projecting the estimator's states orthogonally on the designated hyperplane defined by the output measurements, the estimator states eventually converge to the actual states. It is an interesting observation that from a numerical linear algebra point of view, the above described procedure portrays a direct generalization of the (randomized) Kaczmarz method \cite{strohmer2009randomized} to the situation that $N$ linear equations $Ax_i=b_i$, coupled through the fact that the $N$ different right-hand sides are given in a random order that is undisclosed to us, have to be solved, as already discussed in \cite{zeng2016diss}. An interesting open problem in this novel setup for the Kaczmarz method is to find optimal (possibly random) choices of angles that yield the fastest convergence.
Regarding the online estimation scheme, in Figure~\ref{fig:discrete_online_mhe} we illustrate the result of tracking the position of a discrete ensemble of double integrators from noisy measurements using a moving horizon scheme with a horizon length $T_H = 1$.

\begin{figure}[htp!]
\vspace{-0.3cm}
	\centering
  \includegraphics[width=0.395\textwidth]{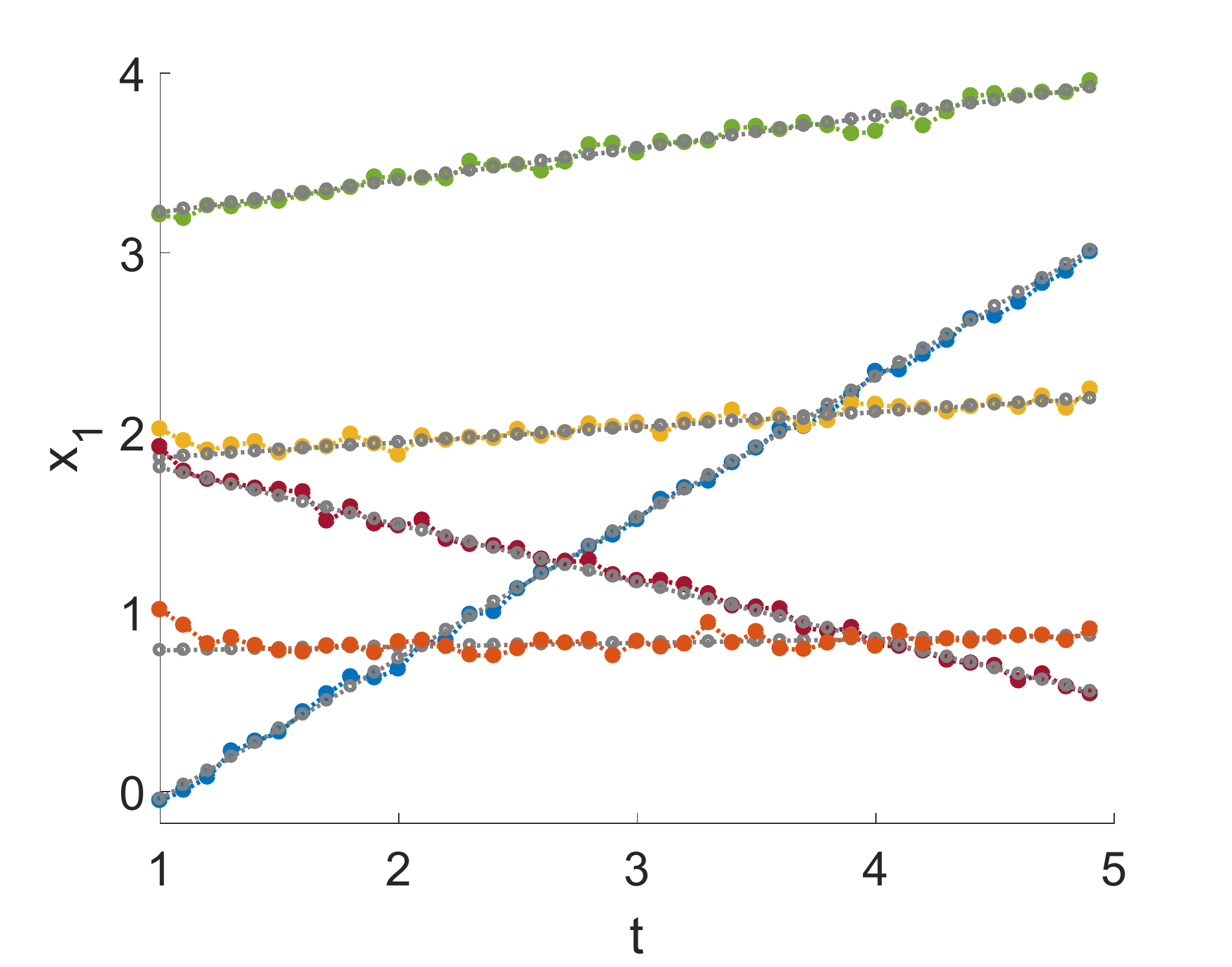}
\vspace{-0.07cm}
	\caption{Reconstruction of the tracts from noisy output measurements $y + n$ where $n \sim N(\mu=0,\sigma^2 = 0.2^2)$ in an online fashion. The grey dotted lines show the actual tracks without noise. The colored tracks show the reconstruction from the observer. In the underlying estimation in the two-dimensional state space, the correction is based on 11 (noisy) recorded measurements in the estimation horizon $[4,5]$, of which only 5 randomly chosen directions are utilized for actually carrying out the Kaczmarz steps.}
	\label{fig:discrete_online_mhe}
\end{figure}
\begin{figure}[htp!]
	\centering
  \includegraphics[width=0.395\textwidth]{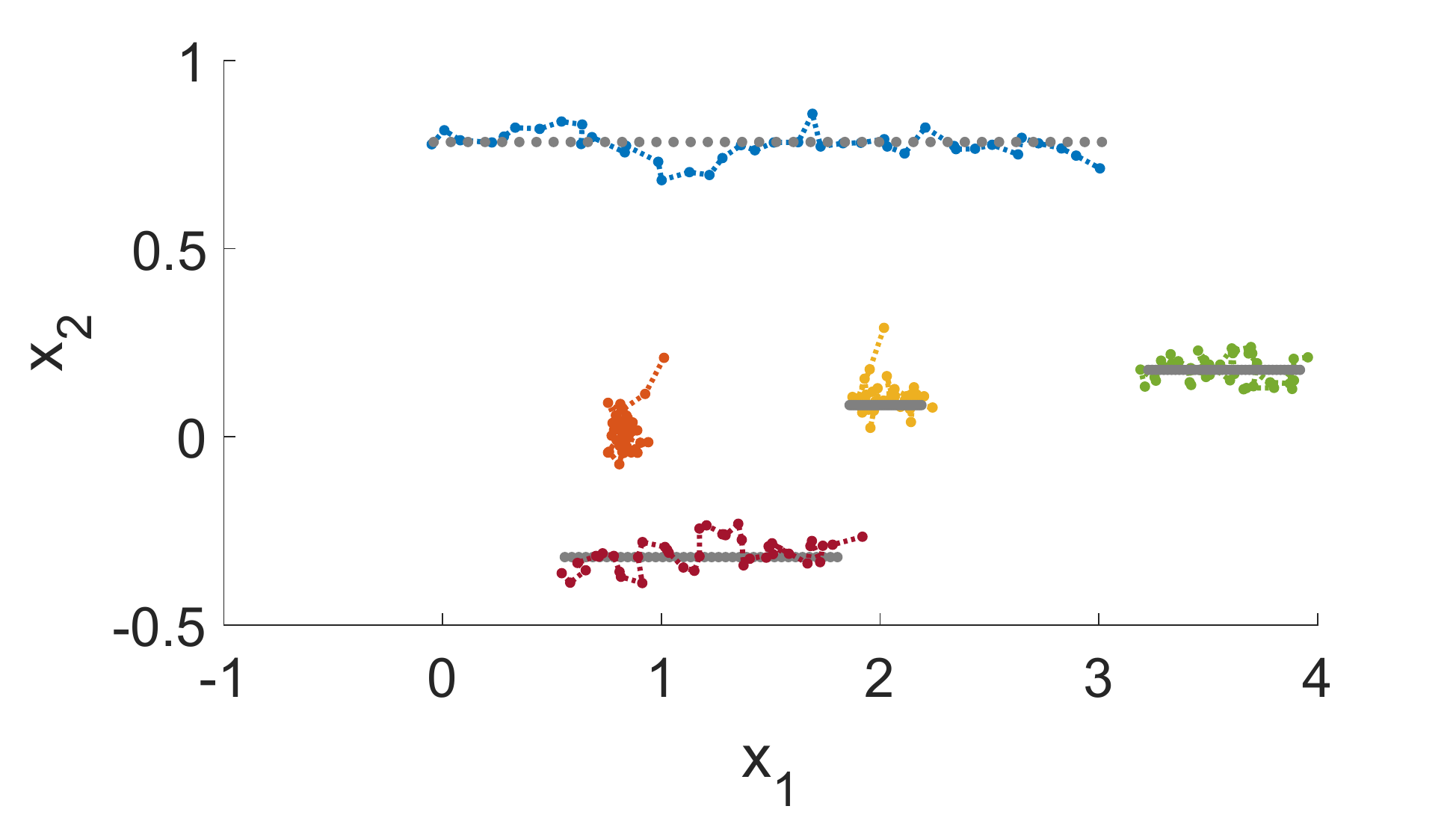}
\vspace{-0.15cm}
	\caption{Illustration of the tracking of the targets in state space. While the estimates for the velocities fluctuate more prominently, the estimates for the position are much more precise due to the fact that position, in contrast to velocity, can be measured directly, and therefore also directly corrected.}
	\label{fig:double_integrator_learning_velocity}
\end{figure}

We note that no ``flipping'' of the position estimates is occurring at the intersection, i.e.\ no sample point that has been tracking one system starts to track another point associated to another system and vice versa. The reason for this is that in the moving horizon approach, by having a time series of measurements, the dynamical component is explicitly taken into account in the estimation. Intuively speaking, by looking at multiple time points, rather than a single time point, and taking our knowledge of the dynamics into account (which happens to be a double integrator), we can also obtain estimates for the velocities, which are used to distinguish the different (anonymized) systems. In fact, it is only by having a horizon of past measurements that corrections in the $x_2$-direction can be achieved in our presented scheme, cf.\ the correction mechanism shown in Figure~\ref{fig:initial_state_estimation_discrete}. The correct ``learning'' of the velocities is illustrated in Figure~\ref{fig:double_integrator_learning_velocity}.

To summarize, the methodology presented in this section for the class of discrete ensembles, which was straightforwardly derived from our sample-based study of the continuous case, provides a significantly improved computational method for the discrete case, which before was handled by a clustered least squares approach in \cite{zeng2017tac}, and was thus limited to problems with about ten agents. With the new approach, it is easily possible to consider problems with hundreds of agents or even more without any difficulties at all. Moreover, if not much emphasis is put on the fact that the tracking takes place on the level of individual systems, i.e.\ one is only interested in a tracking of the population and not a very accurate tracking of individual systems, this method can also be directly applied for large-scale systems, yielding a second simplistic method for problem sizes similar to those in the continuous formulation. A direct simulation example for tracking a population of $N=10^4$ particles is shown in Figure~\ref{fig:tracking_distribution_discrete}.

\begin{figure}[htp!]
\vspace{-0.5cm}
	\centering
	\includegraphics[width=0.45\textwidth]{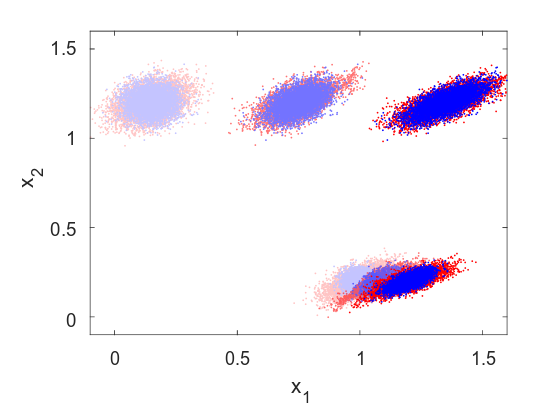}
\vspace{-0.23cm}
	\caption{Three successive predictions at the three time points $t = 0.5, 1.0, 1.5$ using the method for discrete ensembles, but otherwise the same setup as in the earlier example. While one can see that the estimator is able to track the actual ensemble, the convergence rate is slower.
 }
	\label{fig:tracking_distribution_discrete}
\end{figure}

While one can infer that the estimator is able to eventually track the actual ensemble in an acceptable manner, the convergence rate is visibly slower than that in Figure~\ref{fig:MHE_double_integrators}, still only providing a rather coarse estimation in the third estimation step. Furthermore, the computational time is longer compared to the optimal mass transport approach. This is because the method obtained from the study of discrete ensembles involves a sorting of $N$ numbers at each correction step, whereas the optimal mass transport problem formulation does not scale with the actual number of systems $N$, but the number of bins that one chooses. There is, however, a simple remedy in that one can (randomly) choose a subsample of smaller size to speed up the overall estimation process, which would result in a simple to implement, heuristic method for obtaining a quick, first rough estimate.

\section{Conclusions and Outlook}
\label{sec:conclusions}
In the present paper, a first sample-based treatment of the estimation and observation problems associated to the recently emerging class of ensembles of dynamical systems was presented in an introductory manner. The sample-based approach completely circumvents the route over parameterizing the unknown nonparametric probability distribution, which is common to all previous approaches and a crucial aspect, as it all previously considered algorithms to problem setups in which the state-space is low-dimensional.

The starting point for establishing a sample-based approach is the premise of strictly using a set of points in state space as a means to describe / track a distribution rather than to use other approximations such as histograms or more general kernel functions. The main challenge then was to devise an iterative strategy that operates by manipulations on the set of points which would eventually result in the convergence of the set of points to a configuration that could very well be a set of samples from the distribution of interest. From a conceptual point of view, a main result of this paper is the demonstration that optimal mass transport problems, as well as the classical Cram\'{e}r-Wold device, when viewed through the lens of statitics, constitute crucial links in the endeavor to derive \emph{sample-based} population observers.

A key feature of the correction scheme is the interesting two-layer structure that promotes a very basic and simple implementation: The corrective measures for the set of points is computed in a global fashion, based on population-level mismatches, but is eventually implemented on the level of individual particles by feeding population-level data to the individual particles, which compute their own correction by implementing a simple randomized strategy. As a prototype model for the more general scheme portrayed in a two-dimensional state-space, we may consider the system
\begin{align*}
\dot{x}(t) = \begin{pmatrix} \cos(\alpha(t)) \\ \sin(\alpha(t)) \end{pmatrix} (\cos(\alpha(t)) \;\, \sin(\alpha(t)) ) (x_{\text{ref}}(t)-x(t)),
\end{align*}
where again the reference signal of the individual systems $x_{\text{ref}}(t)$ is obtained from population-level considerations and could differ for different systems in the population. Intuitively, this correction scheme can in fact be very naturally thought of along the lines of the process of raking leaves together using rake strokes from several different directions, as we saw in the many illustrations. An interesting open problem in this regard is to derive optimal sequences of angles, possibly formulated in a stochastic framework, that yield a fast convergence for arbitrary configurations of sample points.

It was also shown how the discrete version of the ensemble observability problem can be treated almost as a corollary of the established novel results for the continuous ensemble observer problem. We conclude that presented methodology yields a general and coherent framework for the computational ensemble observability problem.

\newpage
\bibliographystyle{IEEEtran}
\bibliography{references}

\end{document}